\documentclass[11pt]{amsart}
  \usepackage{amsmath, amsfonts, amssymb,amsthm,mathrsfs}
\usepackage{cases} 
\newtheorem{theorem}{Theorem}
\newtheorem{lemma}{Lemma}
\newtheorem{cor}{Corollary}

\theoremstyle{definition}

\newtheorem{defi}[theorem]{Definition}
\newtheorem{remark}[theorem]{Remark}

\newcommand{\C}{\mathbb{C}}

\newcommand{\N}{\mathbb{N}}

\newcommand{\PP}{\mathbb{P}}

\newcommand{\ab}{{\bf a}}

\begin{document}

\title{{Normal family of meromorphic mappings and Big Picard's theorem}}

\author{Nguyen Van Thin and Wei Chen}
\address{Thai Nguyen University of Education, Department of Mathematics, Luong Ngoc Quyen  street, Thai Nguyen city, Thai Nguyen, Viet Nam.}
\address{School of Science, Chongqing University of Posts and Telecommunications, Chongqing, China}
\email{thinmath@gmail.com}
\email{chenwei198841@126.com}
\thanks{Corresponding author: Nguyen Van Thin}


\begin{abstract}
In this paper, we prove some results in normal family of meromorphic mappings intersecting with moving hypersurfaces. As some applications, we establish some results for normal mapping and extension of holomorphic 
mappings. A our result is strongly extended the Montel's normal criterion in the case several variables which is due to Tu in
 [Proc. Amer. Math. Soc. 127, 1039-1049, 1999]. Our results are also strongly extended the results of Tu-Li  in [Sci. China Ser. A.  48, 355-364, 2005] and 
[J. Math. Anal. Appl.  342, 629-638, 2006].
\end{abstract}

\baselineskip=16truept 
\maketitle 
\pagestyle{myheadings}
\markboth{}{}

\section{Introduction and main results}

\def\theequation{1.\arabic{equation}}
\setcounter{equation}{0}

One of the most important developments in complex analysis in the 20th century is the so-called Nevanlinna theory dealing with the value distribution of entire
 and meromorphic functions. In 1933, H. Cartan \cite{CA} generalized the Nevanlinna's result to holomorphic curve in projective space which is now called the
 Nevanlinna-Cartan theory. Since that time, this problem has been studied intensively by many authors. Nevanlinna-Cartan theory has found various applications
 in complex analysis and geometry complex such as uniqueness set theory, normal family theory, extension of holomorphic mapping and the property hyperbolic
 of algebraic variety. In this work, we continue to study the application of Nevanlinna-Cartan theory in normal family, normal mapping and extension of
 holomorphic mappings.

 Let $\mathscr F$ be a family  of meromorphic functions defined on a domain $\mathfrak D$ in the complex plane $\mathbb C.$ $\mathscr F$ is said to be normal 
on $\mathfrak D$ if every sequence functions of $\mathscr F$ has a subsequence which converges uniformly on every compact subset of $\mathfrak D$ with 
respect to the spherical metric to a meromorphic function or identically $\infty$ on $\mathfrak D.$ Perhaps the most celebrated criteria for normality in 
one complex variable is the following Montel-type theorems related to Picard's theorem of value distribution theory.

\begin{theorem}\label{th1}\cite{M1} Let $\mathscr F$ be a family of meromorphic functions on a domain $\mathfrak D$ in the complex plane $\mathbb C.$ 
Suppose that there exist three distinct points $w_1, w_2$ and $w_3$ on the Riemann sphere such that $f(z)-w_i\; (i=1, 2, 3)$ has no zero on $\mathfrak D$ for 
each $f\in F.$ Then $\mathscr F$ is a normal family on $\mathfrak D.$
\end{theorem}

\begin{theorem}\label{th2}\cite{M2} Let $\mathscr F$ be a family of meromorphic functions on a domain $\mathfrak D$ of the complex plane $\mathbb C.$ 
Suppose that there exist distinct points $w_1, w_2, \dots, w_q\; (q\ge 3)$ on the Riemann sphere such that $f(z)-w_i$ has no zero with multiplicities less
 than $m_i \;(i=1, 2, \dots, q)$ on $\mathfrak D$ for each $f\in \mathscr F,$ where $m_i\; (i=1, 2, \dots, q)$ are $q$ fixed positive integers and may be $\infty,$
 with $\sum_{j=1}^{q}\dfrac{1}{m_j}<q-2.$ Then $\mathscr F$ is a normal family on $\mathfrak D.$
\end{theorem}

In the case of higher dimension, the notion of normal family has proved its importance in geometric function theory of several complex variables. Fujimoto
 \cite{F}, Green \cite{MG} and Nochka \cite{EN} established some Picard-type theorems of value distribution theory for holomorphic mappings from 
$\mathbb C^m$ to $\mathbb P^n(\mathbb C),$ related to intersection multiplicity of value distribution theory. By starting from Green's and Nochka's Picard-type
 theorems and using the heuristic principle obtained by Aladro and Krantz \cite{GK}, Tu \cite{T1} extended Theorem \ref{th1} and Theorem \ref{th2} to the
 case of families of holomorphic mappings of a domain $\mathbb D$ in $\mathbb C^m.$ 

Let $A$ be a nonempty open subset of a domain $\mathbb D$ in $\mathbb C^m$ such that $S=\mathbb D\setminus A$ is an analytic set in $\mathbb D.$ Let 
$f: A \to \mathbb P^n(\mathbb C)$ be a holomorphic mapping. Let $U$ be a nonempty connected open subset of $\mathbb D.$ A holomorphic mapping 
$\tilde f$ from $U$ into $\mathbb C^{m+1}$ is said to be a representation of $f$ on $U$ if $f(z)=\rho(\tilde f(z))$ for all $z\in U\cap A-{\tilde f}^{-1}(0),$ 
where $\rho:\mathbb C^{m+1}\to \mathbb P^{n}(\mathbb C)$ is the canonical projection. A holomorphic mapping $f: A\to \mathbb P^{n}(\mathbb C)$ is 
said to be a meromorphic mapping from $\mathbb D$ into $\mathbb P^{n}(\mathbb C)$ if for each $z\in \mathbb D,$ there exists a representation of $f$ on 
some neighborhood of $z$ in $\mathbb D.$

Let $f$ be a meromorphic mapping  of a domain $\mathbb D$ in $\mathbb C^m$ into $\mathbb P^{n}(\mathbb C).$ Then for any 
$a\in \mathbb D,$ $f$ always has an reduced representation $\tilde f(z)=(f_0(z), \dots, f_{n}(z))$ on some neighborhood $U$ of $a$ in $\mathbb D,$ 
which means that each $f_i$ is a holomorphic function on $U$ and that $f(z)=(f_0(z):f_1(z):\dots:f_n(z))$ outside the analytic set 
$I(f):=\{z\in U: f_0(z)=f_1(z)=\dots=f_n(z)=0\}$ of codimension at least $2.$ Obviously a meromorphic mapping from $\mathbb D$ into
 $\mathbb P^{n}(\mathbb C)$ is a holomorphic mapping from $\mathbb D$ into $\mathbb P^{n}(\mathbb C)$ if and only if $I(f)=\emptyset.$ Let $f$ be a 
holomorphic mapping from $\mathbb D$ into $\mathbb P^{n}(\mathbb C)$ and 
 $\tilde{f }$ be a reduced of $f$ on $U.$ Thus $n+1$ meromorphic mappings
$$\tilde f_i=\left(\dfrac{f_0}{f_i},\dots, \dfrac{f_n}{f_i}\right): U\to \overline {\mathbb C}^{n+1},i=0,\dots, n$$
associated a reduced representation $\tilde f$ of $f$  on $U$ do not depend on the representation $\tilde f$, which will be used lately.

\begin{defi}\label{dn2}\cite{GK} Let $\mathbb D$ be a domain in $\mathbb C^m.$ Let $F$ be a family of holomorphic mappings from
 $\mathbb D$ to a compact complex manifold $M$. This family $F$ is said to be a (holomorphically) normal family on $\mathbb D$ if any sequence 
in $F$ contains a subsequence which converges uniformly on compact subsets of $\mathbb D$ to a holomorphic mapping from $\mathbb D$ into $M.$
\end{defi}
\begin{defi}\label{dn3}\cite{F1}
 A sequence $\{f_p\}$ of meromorphic mappings from $\mathbb D$ to $\mathbb P^{n}(\mathbb C)$ is said to converge meromorphically on
 $\mathbb D$ to a meromorphic mapping $f$ if and only if, for any $z\in \mathbb D,$ each $f_p$ has a reduced representation
$\tilde f_p= (f_{0,p},f_{1,p},\dots,f_{n,p})$ on some fixed neighborhood $U$ of $z$ in $\mathbb D$ such that $\{f_{i,p}\}_{p=1}^{\infty}$
 converges uniformly on compact subsets of $U$ to a holomorphic function $f_i\; (0\le i \le n)$ on $U$ with the property that
 $f = (f_0: f_1: \dots : f_n)$ is a representation of $f$ on $U,$ not necessarily a reduced representation.
\end{defi}
\begin{defi}\label{dn4}\cite{F1}
 Let $F$ be a family of meromorphic mappings of $\mathbb D$ into $\mathbb P^{n}(\mathbb C).$ One may call $F$ a meromorphically normal family 
on $\mathbb D$ if any sequence in $F$ has a meromorphically convergent subsequence on $\mathbb D.$
\end{defi}

\begin{defi}\label{dn4a}\cite{F1}
 A sequence $\{f_p\}$ of meromorphic mappings from $\mathbb D$ into $\mathbb P^{n}(\mathbb C)$ is said to be quasi-regular on $\mathbb D$ 
if and only if for any $z\in \mathbb D$ has a neighborhood $U$ with property that $\{f_p\}$ converges compactly on $U$ outside a nowhere dense analytic subset $S$ of $U,$ i.e., for any domain $G\subset U\setminus{S}$ such that closure $\overline G$ of $G$ in $U\setminus{S}$ is a compact subset of $U,$ there is some $p_0$ such that $I(f_p)\cap G=\emptyset$ for all $p\ge p_0$ and $\{f_p|_{G}\}_{p\ge p_0}$ converges uniformly on $G$ to a holomorphic mapping of $G$ into $\mathbb P^{n}(\mathbb C).$
\end{defi}

\begin{defi}\label{dn5a}\cite{F1}
 Let $F$ be a family of meromorphic mappings of $\mathbb D$ into $\mathbb P^{n}(\mathbb C).$  $F$ is said to be quasi-normal family on $\mathbb D$ if any sequence in $F$ has a subsequence so as to be quasi-regular on $\mathbb D.$
\end{defi}
For the detailed discussion about quasi-normal family, see Fujimoto \cite{F1}.

In 2005, Tu and Li extended the result of Tu \cite{T1} for moving target and obtained the result as follows:

\begin{theorem}\label{th4}\cite{TL}
Let $\mathcal F$ be a family of holomorphic mappings of a domain $\mathbb D$ in $\mathbb C^m$ into $\mathbb P^n(\mathbb C)$ and let $H_1, \dots, H_q$ 
be $q\; (q\ge 2n+1)$ moving hyperplanes in $\mathbb P^n(\mathbb C)$ located in pointwise general position such that each $f$ in $F$ intersects $H_j$ 
on $\mathbb D$ with multiplicity at least $m_j\; (j=1, \dots, q),$ where $m_1, \dots, m_q$ are fixed positive integers and may be $\infty,$ with 
$\sum_{j=1}^{q}\dfrac{1}{m_j}<\dfrac{q-n-1}{n}.$ Then $\mathcal F$ is a normal family on $\mathbb D.$
\end{theorem}

In 2006, Tu and Li proved some results as follows for normal mapping and Big's Picard Theorem:

\begin{theorem}\label{th5}\cite{TL1}
Let $S$ be an analytic subset of a domain $\mathbb D$ in $\mathbb C^m$ with $\dim_{\mathbb C} S\le m-2.$ Let $f$ be a holomorphic mapping from 
$\mathbb D-S$ into $\mathbb P^n(\mathbb C).$ Let $H_1, \dots, H_q$ be $q$ $(\ge 2n+1)$ moving hyperplanes in $\mathbb P^n(\mathbb C)$ 
located in pointwise general position such that $f$ intersects $H_j$ on $\mathbb D-S$ with multiplicity at least $m_j$ $(j=1, \dots, q),$ where $m_1, \dots, m_q$
 are positive integers and may be $\infty,$ with $\sum_{j=1}^{q}\dfrac{1}{m_j}<\dfrac{q-n-1}{n}.$ Then the holomorphic mapping 
$f$  from $\mathbb D-S$ into $\mathbb P^n(\mathbb C)$ extends to a holomorphic mapping from $\mathbb D$ into $\mathbb P^n(\mathbb C).$
\end{theorem}

\begin{theorem}\label{th6}\cite{TL1}
Let $S$ be an analytic subset of a domain $\mathbb D$ in $\mathbb C^m$ with codimension one, whose singularities are normal crossings. Let 
$f$ be a holomorphic mapping from $\mathbb D-S$ into $\mathbb P^n(\mathbb C).$  Let $H_1, \dots, H_q$ be $q$ $(\ge 2n+1)$ moving hyperplanes in
 $\mathbb P^n(\mathbb C)$ located in pointwise general position such that $f$ intersects $H_j$ on $\mathbb D-S$ with multiplicity at least
 $m_j$ $(j=1, \dots, q),$ where $m_1, \dots, m_q$ are positive integers and may be $\infty,$ with $\sum_{j=1}^{q}\dfrac{1}{m_j}<\dfrac{q-n-1}{n}.$ 
Then the holomorphic mapping $f$  from $\mathbb D-S$ into $\mathbb P^n(\mathbb C)$ extends to a holomorphic mapping from $\mathbb D$ into
 $\mathbb P^n(\mathbb C).$
\end{theorem}

In 2015, Dethloff-Thai-Trang \cite{DTT} generalized the result of Tu-Li \cite{TL} for meromorphic mapping intersecting a class hypersurfaces in weak general 
position. Detail, they consider the class hypersurfaces with contruction as follows:
Let $H_{\mathbb D}$ be the ring of holomorphic functions on $\mathbb D$ and let $H_{\mathbb D}[x_0, \dots, x_n]$ be the set of homogeneous polynomials of variables ${\bf x}=(x_0, \dots, x_n)$ over $H_{\mathbb D}$. Take $Q\in H_{\mathbb D}[x_0, \dots, x_n]\setminus\{0\}$ with $\deg(Q)=d$ and write
$$Q({\bf x})=Q(x_0, \dots, x_n) = \sum\limits_{k=0}^{n_d}a_k {\bf x}^{I_k}= \sum\limits_{k=0}^{n_d}a_k x_0^{i_{k0}}\dots x_n^{i_{kn}},$$
where $a_k\in H_{\mathbb D}$, $I_k=(i_{k0},\dots ,i_{kn})$ with $|I_k|=i_{k0}+\dots +i_{kn} =d$ for $k=0,\dots,n_d$ and $n_d= \binom{n+d}{n}-1$. Thus each $z\in \mathbb D$ corresponds to an element $Q_z$ in $\mathbb C[x_0, \dots, x_n]$ defined by
$$Q_z({\bf x}) = \sum\limits_{k=0}^{n_d}a_k(z) {\bf x}^{I_k},$$
and hence a hypersurface
\begin{equation*}
    D_z:= \{{\bf x} \in \mathbb C^{n+1}: Q_z({\bf x}) = 0\}
\end{equation*}
is associated to $z$ or $Q_z$. The correspondence $D$ from $\mathbb D$ into hypersurfaces is called a {\it moving hypersurface} in $\mathbb P^{n}(\mathbb C)$ defined by $Q$. Let $\ab= (a_0,\dots,a_{n_d})$ be the
vector function associated to $D$ (or $Q$). In this paper, we only consider homogeneous polynomials $Q$ over $H_{\mathbb D}[x_0, \dots, x_n]$  such that
 coefficients of $Q$ have no common zeros, so that the hypersurface $D_z$ is well defined for each $z\in \mathbb D$.

 Let $P_0, \dots, P_n$ be $n+1$ moving homogeneous polynomials of common degree in $H_{\mathbb D}[x_0, \dots, x_n].$ Denote by $S(\{P_i\}_{i=0}^{n})$ the set of all homogeneous not identically zero polynomials $Q=\sum_{i=0}^{n}b_iP_i, b_i\in H_{\mathbb D}.$ Let $T_0, \dots, T_n$ be moving hypersurfaces in $\mathbb P^n(\mathbb C)$ with common degree, where $T_i$ is defined by the (not identically zero) polynomial $P_i\; (0\le i \le N).$ Denote by $\overset{\sim}S(\{T_i\}_{i=0}^{n})$ the set of all moving hypersurfaces in $\mathbb P^n(\mathbb C)$ which are defined by $Q\in S(\{P_i\}_{i=0}^{n}).$ 

Let $\{Q_j=\sum_{i=0}^{n}b_{ij}P_i\}_{j=1}^{q}$ be $q\;(q\ge n+1)$ homogeneous polynomials in $S(\{P_i\}_{i=0}^{n}).$ We say that $\{Q_j\}_{j=1}^{q}$ are located pointwise in general position in $S(\{P_i\}_{i=0}^{n})$ if 
$$ \text{det}(b_{ij_k})_{0\le k,i\le n}\ne 0 $$ 
for all $1\le j_0<\dots<j_n\le q$ and all $z\in \mathbb D.$

For homogeneous coordinate $x=(x_0:\dots:x_n)$ in $\mathbb P^n(\mathbb C),$ we denote ${\bf x}=(x_0,\dots,x_n).$ Before state the results of Dethloff-Thai-Trang, we need some definitions as follows:
\begin{defi}\label{dn1}
Let $D_j$ be moving hypersurfaces in $\mathbb P^n(\mathbb C)$ of degree $d_j$ which is defined by polynomial homogeneous $Q_j\in H_{\mathbb D}[x_0,\dots,x_n],j=1,\dots,q.$ We say that moving hypersurfaces $\{D_j\}_{j=1}^{q}$ $ (q \ge n+1)$ in $\mathbb P^n(\mathbb C)$ are located in (weakly) general position if there exists $z\in \mathbb D$ such that, for any $1\le j_0 < \dots < j_n \le q,$ the system of equations
$$ \begin{cases}
Q_{j_i}(z)(x_0, \dots, x_n)=0\\
\hspace{1cm}0\le j\le n
\end{cases}$$
has only the trivial solution ${\bf x}= (0, \dots, 0)$ in $\mathbb C^{n+1}.$
\end{defi}
This is equivalent to
\begin{align*}
D(Q_1, \dots, Q_q)(z)=\prod_{1\le j_0<\dots<j_n\le q}\inf_{||{\bf x}||=1}(|Q_{j_0}(z)({\bf x})|^2+\dots+|Q_{j_n}(z)({\bf x})|^2)>0,
\end{align*}
where ${\bf x}=(x_0, \dots, x_n),$ $Q_j(z)({\bf x}) = \sum_{|I|=d_j}a_{jI}(z){\bf x}^{I}$ and $||{\bf x}||=(\sum_{j=0}^{n}|x_j|^{2})^{1/2}.$

Let $f$ be a meromorphic mapping from a domain $\mathbb D$ in $\mathbb C^m$ into $\mathbb P^n(\mathbb C).$ We denote $D_f$ by the hypersurface in $\mathbb P^{n}(\mathbb C)$ depending on $f,$ and $Q_f$ by the homogeneous polynomial depending on $f,$ where $Q_f$ is homogeneous polynomial defining by $D_f.$ With that notation, Dethloff-Thai-Trang proved the results as follows:
\begin{theorem}\label{th8}\cite{DTT}
Let $\mathcal F$ be a family of holomorphic mappings of a domain $\mathbb D$ in $\mathbb C^m$ into $\mathbb P^n(\mathbb C).$ Let $q\ge 2n+1$ be 
a positive integer. Let $m_1, \dots, m_q$ be positive integers or $\infty$ such that
$$ \sum_{j=1}^{q}\dfrac{1}{m_j}<\dfrac{q-n-1}{n}.$$
Suppose that for each $f\in F,$ there exist $n+1$ moving hypersurfaces $T_{0,f},$ $ \dots, T_{n,f}$  in $\mathbb P^n(\mathbb C)$ of common degree 
and that there exist $q$ moving hypersurfaces $D_{1,f},$ $ \dots, D_{q,f}$ in $\overset{\sim}S(\{T_{i,f}\}_{i=0}^{n})$ such that the following 
conditions are satisfied:

\noindent $(i)$ For each $0\le i \le n,$ the coefficients of the homogeneous polynomials $P_{i,f}$ which define the $T_{i,f}$ are bounded above 
uniformly for all $f$ in $F$ on compact subsets of $\mathbb D$ for all $1\le j\le q,$ the coefficients $b_{ij}(f)$ of the linear combinations of the
 $P_{i,f}, i = 0, \dots, n,$ which define the homogeneous polynomials $Q_{j,f}$  are bounded above uniformly for
all $f$ in $\mathcal F$ on compact subsets of $\mathbb D,$ where $Q_{j,f} \in S(\{P_{i,f}\}_{i=0}^{n})$ is a homogeneous polynomial defined the $D_{j,f},$ 
and for any fixed $z\in \mathbb D,$
$$ \inf_{f\in F}D(Q_{1,f}, \dots, Q_{q,f})(z) > 0.$$

\noindent $(ii)$ Assume that $f$ intersects $D_{j,f}$ with multiplicity at least $m_j$ for each $1\le j\le q,$ for all $f\in \mathcal F.$

Then $\mathcal F$ is a holomorphically normal family on $\mathbb D$.
\end{theorem}

\begin{theorem}\label{th9}\cite{DTT}
Let $\mathcal F$ be a family of meromorphic mappings of a domain $\mathbb D$ of $\mathbb C^m$ into $\mathbb P^{n}(\mathbb C).$ Let
 $q\ge 2n+1$ be a positive integer. Suppose that for each $f\in \mathcal F,$ there exist $n+1$ moving hypersurfaces 
$T_{0,f}, \dots, T_{n,f}$  in $\mathbb P^n(\mathbb C)$ of common degree and that there exist $q$ moving hypersurfaces 
$D_{1,f}, \dots, D_{q,f}$ in $\overset{\sim}S(\{T_{i,f}\}_{i=0}^{n})$ such that the following conditions are satisfied:

\noindent $(i)$ For each $0\le i \le n,$ the coefficients of the homogeneous polynomials $P_{i,f}$ which define the $T_{i,f}$ are bounded above uniformly for 
all $f$ in $\mathcal F$ on compact subsets of $\mathbb D$ for all $1\le j\le q,$ the coefficients $b_{ij}(f)$ of the linear combinations of the
 $P_{i,f}, i = 0, \dots, n,$ which define the homogeneous polynomials $Q_{j,f}$  are bounded above uniformly for
all $f$ in $\mathcal F$ on compact subsets of $\mathbb D,$ where $Q_{j,f} \in S(\{P_{i,f}\}_{i=0}^{n})$ is a homogeneous polynomial defined the 
$D_{j,f},$ and for any $\{f_p\}\subset \mathcal F,$ there exists $z\in \mathbb D$ (which may depend on sequence) such that
$$ \inf_{p\in \mathbb N}D(Q_{1,f_p}, \dots, Q_{q,f_p})(z) > 0.$$

\noindent $(ii)$ For any fixed compact $K$ of $\mathbb D,$ the $2(m-1)$-dimensional Lebesgue areas of $f^{-1}(D_{k,f})\cap K$ $(1\le k\le n+1)$ 
counting multiplicities are bounded above for all $f\in \mathcal F$ (in particular, $f(\mathbb D)\not\subset D_{k,f} (1\le k\le n+1)).$

\noindent $(iii)$ There exists a closed subset $S$ of $\mathbb D$ with $\Lambda^{2m-1}(S) = 0$ such that for any fixed compact subset 
$K$ of $\mathbb D-S,$ the $2(m-1)$-dimensional Lebesgue areas of
$$  \{z\in supp \nu_{Q_{k,f}(\tilde f)}: \nu_{Q_{k,f}(\tilde f)}(z)<m_k\}\cap K \; (n+2\le k\le q)$$
ignoring multiplicities for all $f\in \mathcal F$ are bounded above, where $\{m_k\}_{k=n+2}^{q}$ are fixed positive integers or $\infty$ with
$$ \sum_{k=n+2}^{q}\dfrac{1}{m_k}<\dfrac{q-n-1}{n}.$$

Then $\mathcal F$ is a meromorphically normal family on $\mathbb D.$
\end{theorem}

 Let $\mathcal D = \{D_1,\dots,D_q\}$ be a collection of arbitrary hypersurfaces and $Q_j$ be the homogeneous polynomial in $\mathbb C[x_0, \dots, x_n]$
 of degree $d_j$ defining $D_j$ for $ j =1,\dots,q$. Let $m_{\mathcal D}$ be the least common multiple of the $d_j$ for $j=1,\dots,q$ and denote
$$n_{\mathcal D}= \binom{n+m_{\mathcal D}}{n}-1. $$
For $j=1,\dots,q,$ we set $Q^*_j = Q_j^{m_{\mathcal D}/d_j}$ and let $\ab^*_j$ be the vector associated with $Q^*_j$.

We recall the {\it lexicographic ordering} on the $(n+1)$-tuples of natural numbers: let $J=(j_0,\dots,j_n), I=(i_0,\dots,i_n) \in \N^{n+1}$, $J<I$ if 
and only if for some $b \in \{0,\dots,n\}$ we have $j_l = i_l$ for $l<b$ and $j_b < i_b$. With the $(n+1)$-tuples $I = (i_0,\dots, i_{n})$ of non-negative integers, we denote $|I|:= \sum_j i_j.$

Let $(x_0:\dots:x_n)$ be a homogeneous coordinates in $\mathbb P^n(\mathbb C)$ and let $\{I_0,\dots,I_{n_{\mathcal D}}\}$ be a set of $(n+1)$-tuples 
such that $|I_j|= m_{\mathcal D}, \ j =0,\dots,n_{\mathcal D},$ and $I_i < I_j$ for $i < j \in \{0,\dots,n_{\mathcal D}\} $. We denote ${\bf x}=(x_0,\dots,x_n) \in \mathbb C^{n+1},$ and write ${\bf x}^{I}$ as $x_0^{i_0}\dots x_n^{i_n}$ where $I = (i_0,\dots,i_n) \in \{I_0,\dots,I_{n_{\mathcal D}}\}$. Then we may order the set of monomials of degree $ m_{\mathcal D}$ as $\{ {\bf x}^{I_0}, \dots, {\bf x}^{I_{n_{\mathcal D}}}\}$.

Denote $$\varrho_{m_{\mathcal D}}: \mathbb P^n(\mathbb C) \to \mathbb P^{n_{\mathcal D}}(\mathbb C)$$ be the {\it Veronese embedding of degree 
$m_{\mathcal D}$}.  Let $(w_0:\dots:w_{n_{\mathcal D}})$ be the homogeneous coordinates in $\PP^{n_{\mathcal D}}(\mathbb C)$. 
Then $\varrho_{m_{\mathcal D}}$ is given by
$$\varrho_{m_{\mathcal D}}(x) = (w_0({\bf x}):\dots:w_{n_{\mathcal D}}({\bf x})), \ \text{ where } \ w_j(x) = {\bf x}^{I_j}, \ j=0,\dots,n_{\mathcal D}.$$
For any hypersurface $D_j\in \{D_1,\dots,D_q\}$ and $\ab_j = (a_{j0},\dots,a_{jn_{\mathcal D}})$ be the vector associated with $Q_j^{*}$, we set 
$$L_j=Q_j^{*}=  a_{j0}w_0 + \dots + a_{j{n_{\mathcal D}}}w_{n_{\mathcal D}}.$$
Then $L_j$ is a linear form in $\mathbb P^{n_{\mathcal D}}(\mathbb C)$. Let $H_j^{*}$ be a hyperplane in $\mathbb P^{n_{\mathcal D}}(\mathbb C)$, 
which is  defined by $L_j$, we say that the hyperplane $H_j^{*}$ associated with $Q_j^{*}$  (or $D_j$). Hence for the collection of hypersurfaces 
$\{D_1,\dots, D_q\}$ in $\mathbb P^{n}(\mathbb C)$, we have the collection $\mathcal H^{*} =\{H_1^{*},\dots,H_q^{*}\}$ of associated hyperplanes in
 $\mathbb P^{n_{\mathcal D}}(\mathbb C)$.

 Let $q> n_{\mathcal D}$ be a positive integer. We say that collection $\mathcal D$ is in general position for Veronese embedding in $\mathbb P^n(\mathbb C)$ 
if $\{H_1^{*},\dots,H_q^{*}\} $ is in general position in $\mathbb P^{n_{\mathcal D}}(\mathbb C).$ Hence, the collection $\mathcal D$ is in {\it general position for Veronese embedding} if for any distinct
 $i_0,\dots,i_{n_{\mathcal D}} \in \{1,\dots,q\}$, the vectors $\ab^*_{i_0},\dots,\ab^*_{i_{n_{\mathcal D}}}$ have rank $n_{\mathcal D}.$ It is seen that, 
for hyperplanes, general position for Veronese embedding is equivalent to the usual concept of hyperplanes in general position usual. For hypersurfaces, general 
position for Veronese embedding implies $n_{\mathcal D}$-subgeneral position.

Let $\mathcal D = \{D_1,\dots,D_q\}$ be a collection of arbitrary moving hypersurfaces and $Q_j$ be the homogeneous polynomial in $ H_{\mathbb D}[x_0, \dots, x_n]$ of degree $d_j$ defining $D_j$ for $ j =1,\dots,q$. Let $m_{\mathcal D}$ be the least common multiple of the $d_j$ for $j=1,\dots,q$ and denote
$$n_{\mathcal D}= \binom{n+m_{\mathcal D}}{n}-1. $$
For $j=1,\dots,q,$ we set $Q^*_j = Q_j^{m_{\mathcal D}/d_j}$ and let $\ab^*_j$ be the function vector associated with $Q^*_j.$ We set 
$$L_j=Q_j^{*} =  a_{j0}w_0 + \dots + a_{j{n_{\mathcal D}}}w_{n_{\mathcal D}}.$$
Then $L_j$ is a moving linear form in $\mathbb P^{n_{\mathcal D}}(\mathbb C)$. Let $H_j^{*}$ be a moving hyperplane in 
$\mathbb P^{n_{\mathcal D}}(\mathbb C)$, which is  defined by $L_j$ (or $Q_j^{*}$), we say that the hyperplane $H_j^{*}$ associated with
 $Q_j^{*}$  (or $D_j$). Hence for the collection of moving hypersurfaces $\{D_1,\dots, D_q\}$ in $\mathbb P^{n}(\mathbb C)$, we have the
 collection $\mathcal H^{*} =\{H_1^{*},\dots,H_q^{*}\}$ of associated moving hyperplanes in $\mathbb P^{n_{\mathcal D}}(\mathbb C)$.

\begin{defi}\label{dnm}
Let  $\{D_j\}_{j=1}^{q}$ be moving hypersurfaces in $\mathbb P^n(\mathbb C)$ which are defined by homogeneous 
$Q_j\in H_{\mathbb D}[x_0, \dots, x_n]$ of degree $d_j\; (q \ge n_{\mathcal D}+1),j=1,\dots,q.$ We say that moving hypersurfaces $\{D_j\}_{j=1}^{q}$
 are located in (weakly) general position for Veronese embedding in $\mathbb P^n(\mathbb C)$ if there exists $z\in \mathbb D$ such that, the collection
 $\mathcal H^{*}(z)=\{H_1^{*}(z),\dots,H_q^{*}(z)\}$ of associated hyperplanes in $\mathbb P^{n_{\mathcal D}}(\mathbb C)$ at $z$ are in general 
position in $\mathbb P^{n_{\mathcal D}}(\mathbb C).$
\end{defi}
\begin{remark}\label{rmn}
We see that hypersurfaces $\{D_j\}_{j=1}^{q}$ are located in (weakly) general position for Veronese embedding if there exists $z\in \mathbb D$ such that 
$$D(Q_1^{*}, \dots, Q_q^{*})(z) > 0$$
in $\mathbb P^{n_{\mathcal D}}(\mathbb C).$
\end{remark}

Our results are given as follows:

\begin{theorem}\label{th10}
Let $\mathcal F$ be a family of holomorphic mappings of a domain $\mathbb D$ in $\mathbb C^m$ into $\mathbb P^n(\mathbb C).$ Let
 $q\ge 2n_{\mathcal D}+1$ be a positive integer. Let $m_1, \dots, m_q$ be positive integers or $\infty$ such that
$$\sum_{j=1}^{q}\dfrac{d_j}{m_j}<\dfrac{(q-n_{\mathcal D}-1)m_{\mathcal D}}{n_{\mathcal D}}.$$
Suppose that for each $f\in \mathcal F,$ there exist $q$ moving hypersurfaces $D_{1,f}, \dots, D_{q,f}$ which are defined by homogeneous polynomials 
$Q_{1,f},\dots, Q_{q,f}\in H_{\mathbb D}[x_0, \dots, x_n]$ with degree $d_1,\dots,d_q$ respectively such that the following conditions are satisfied
 on each compact subset $K\subset \mathbb D:$

\noindent $(i)$ For each $1\le j \le q,$ the coefficients of the homogeneous polynomials $Q_{j,f}$ are bounded above uniformly for all $f$ in $\mathcal F$ 
on compact subset $K$ of $\mathbb D.$ Set $m_{\mathcal D}=lcm(d_1, \dots, d_q),$ $n_{\mathcal D}=\binom{n+m_{\mathcal D}}{n}-1,$ and 
$Q_{j,f}^{*}=Q_{j,f}^{m_{\mathcal D}/{d_j}}.$ Suppose that $Q_{j,f}^{*}=\sum_{i=0}^{n_{\mathcal D}}c_{ij}(f)\omega_i,\; (\omega_i={\bf x}^{I_i})$ 
in $H_{\mathbb D}[\omega_0,\dots,\omega_{n_{\mathcal D}}]$ is moving linear form defines the associated hyperplane 
$H_{j,f}^{*}$ of $D_{j,f}\; (j=1,\dots,q)$ and for any $\{f_p\}\subset \mathcal F,$ and for any fixed $z\in \mathbb D,$
 $$ \inf_{p\in \mathbb N}D(Q_{1,f_p}^{*}, \dots, Q_{q,f_p}^{*})(z) > 0.$$

\noindent $(ii)$ Assume that for each $f\in \mathcal F,$ and $\tilde f=(f_0,\dots, f_n)$ is a reduced representation of $f\in \mathcal{F}$  on $\mathbb D$ 
with $ f_i(z)\ne 0$ for 
some $i$ and $z,$ and when $m_j\geq 2$
        $$\sup_{1\leq |\alpha|\leq m_j-1, z \in f^{-1}(D_{j,f})\cap K}\Big|\dfrac{\partial^{\alpha}(Q_{j,f}(\tilde f_{i}))}{\partial z^{\alpha}}(z)\Big| < \infty$$
        hold for all $j\in\{1,\dots,q\}$ and all $f\in \mathcal F.$

Then $\mathcal F$ is a holomorphically normal family on $\mathbb D$.
\end{theorem}
Theorem \ref{th10} is strongly extended the Montel's normal criterion in the case several variables which due to Tu \cite{T1}. From Theorem \ref{th10}, we get the result as follows:
\begin{cor}\label{corth10}
Let $\mathcal F$ be a family of holomorphic mappings of a domain $\mathbb D$ in $\mathbb C^m$ into $\mathbb P^n(\mathbb C).$ Let
 $q\ge 2n_{\mathcal D}+1$ be a positive integer. Let $m_1, \dots, m_q$ be positive integers or $\infty$ such that
$$\sum_{j=1}^{q}\dfrac{d_j}{m_j}<\dfrac{(q-n_{\mathcal D}-1)m_{\mathcal D}}{n_{\mathcal D}}.$$
Suppose that for each $f\in \mathcal F,$ there exist $q$ moving hypersurfaces $D_{1,f}, \dots, D_{q,f}$ which are defined by homogeneous polynomials
 $Q_{1,f},\dots, Q_{q,f}\in H_{\mathbb D}[x_0, \dots, x_n]$ with degree $d_1,\dots,d_q$ respectively such that the following conditions are satisfied  
on each compact subset $K\subset \mathbb D:$

\noindent $(i)$ For each $1\le j \le q,$ the coefficients of the homogeneous polynomials $Q_{j,f}$ are bounded above uniformly for all $f$ in $\mathcal F$ on 
  compact subset $K\subset \mathbb D.$ Set $m_{\mathcal D}=lcm(d_1, \dots, d_q),$ $n_{\mathcal D}=\binom{n+m_{\mathcal D}}{n}-1,$ and
 $Q_{j,f}^{*}=Q_{j,f}^{m_{\mathcal D}/{d_j}}.$ Suppose that $Q_{j,f}^{*}=\sum_{i=0}^{n_{\mathcal D}}c_{ij}(f)\omega_i,\; (\omega_i={\bf x}^{I_i})$ 
in $H_{\mathbb D}[\omega_0,\dots,\omega_{n_{\mathcal D}}]$ is moving linear form defines the associated hyperplane $H_{j,f}^{*}$ of
 $D_{j,f}\; (j=1,\dots,q)$ and for any $\{f_p\}\subset \mathcal F,$ and for any fixed $z\in \mathbb D,$
 $$ \inf_{p\in \mathbb N}D(Q_{1,f_p}^{*}, \dots, Q_{q,f_p}^{*})(z) > 0.$$

\noindent $(ii)$ Assume that for each $f\in \mathcal F,$ $f$ intersects $D_{j,f}$ with multiplicity at least $m_j$ for each $1\le j\le q.$ 

Then $\mathcal F$ is a holomorphically normal family on $\mathbb D$.
\end{cor}
\begin{theorem}\label{th11}
Let $\mathcal F$ be a family of meromorphic mappings of a domain $\mathbb D$ in $\mathbb C^m$ into $\mathbb P^n(\mathbb C).$ Let 
$q\ge 2n_{\mathcal D}+1$ be a positive integer. Suppose that for each $f\in \mathcal F,$ there exist $q$ moving hypersurfaces $D_{1,f}, \dots, D_{q,f}$
 which defined by homogeneous polynomials $Q_{1,f},\dots, Q_{q,f}\in H_{\mathbb D}[x_0, \dots, x_n]$ with degree $d_1,\dots,d_q$ 
respectively such that the following conditions are satisfied:

\noindent $(i)$ For each $1\le j \le q,$ the coefficients of the homogeneous polynomials $Q_{j,f}$ are bounded above uniformly for all $f$ in $\mathcal F$
 on compact subsets of $\mathbb D.$ Set $m_{\mathcal D}=lcm(d_1, \dots, d_q),$ $n_{\mathcal D}=\binom{n+m_{\mathcal D}}{n}-1,$ and
$Q_{j,f}^{*}=Q_{j,f}^{m_{\mathcal D}/{d_j}}.$ Suppose that $Q_{j,f}^{*}=\sum_{i=0}^{n_{\mathcal D}}c_{ij}(f)\omega_i,\; (\omega_i={\bf x}^{I_i})$ 
in $H_{\mathbb D}[\omega_0,\dots,\omega_{n_{\mathcal D}}]$ is moving linear form defines the associated hyperplane $H_{j,f}^{*}$ of $D_{j,f}$ and
 for any $\{f_p\}\subset \mathcal F,$ there exists $z\in \mathbb D$ (which may depend on sequence) such that
 $$ \inf_{p\in \mathbb N}D(Q_{1,f_p}^{*}, \dots, Q_{q,f_p}^{*})(z) > 0.$$

\noindent $(ii)$ For any fixed compact $K$ of $\mathbb D,$ the $2(m-1)$-dimensional Lebesgue areas of 
$f^{-1}(D_{k,f})\cap K$ $(1\le k\le n_{\mathcal D}+1)$ counting multiplicities are bounded above for all $f\in \mathcal F$ 
(in particular, $f(\mathbb D)\not\subset D_{k,f} (1\le k\le n_{\mathcal D}+1)).$

\noindent $(iii)$ There exists a nowhere dense analytic set $S$ in $\mathbb D$ such that for any fixed compact subset $K$ of 
$\mathbb D-S,$ the $2(m-1)$-dimensional Lebesgue areas of
$$  \{z\in \text{supp} \nu_{Q_{k,f}(\tilde f)}: \nu_{Q_{k,f}(\tilde f)}(z)<m_k\}\cap K \; (n_{\mathcal D}+2\le k\le q)$$
ignoring multiplicities for all $f\in \mathcal F$ are bounded above, where $\{m_k\}_{k=n_{\mathcal D}+2}^{q}$ are fixed positive integers or $\infty$ with
$$ \sum_{k=n_{\mathcal D}+2}^{q}\dfrac{d_k}{m_k}<\dfrac{(q-n_{\mathcal D}-1)m_{\mathcal D}}{n_{\mathcal D}}.$$

Then $\mathcal F$ is a quasi-normal family on $\mathbb D.$
\end{theorem}

\begin{theorem}\label{th12}
Let $S$ be an analytic subset of a domain $\mathbb D$ in $\mathbb C^m$ with $\dim_{\mathbb C} S\le m-2.$ Let $f$ be a holomorphic mapping from 
$\mathbb D-S$ into $\mathbb P^n(\mathbb C).$ Let $D_1, \dots, D_q$ be $q$ $(\ge 2n_{\mathcal D}+1)$ moving hypersurfaces in $\mathbb P^n(\mathbb C).$
 Let $Q_j^{*}=\sum_{i=0}^{n_{\mathcal D}}d_{ij}\omega_i,\; (\omega_i={\bf x}^{I_i})$ in $H_{\mathbb D}[\omega_0,\dots,\omega_{n_{\mathcal D}}]$ 
be moving linear form defines the associated hyperplane $H_j^{*}$  of $D_j$ such that for any $z\in \mathbb D:$
 $$D(Q_1^{*}, \dots, Q_q^{*})(z) > 0.$$ 
Let $\tilde f=(f_0,\dots, f_n)$ be a reduced representation of $f$  on $\mathbb D-S$ 
with $ f_i(z)\ne 0$ for some $i$ and $z,$ and when $m_j\geq 2$
        $$\sup_{1\leq |\alpha|\leq m_j-1, z \in f^{-1}(D_{j})}\Big|\dfrac{\partial^{\alpha}(Q_{j}(\tilde f_{i}))}{\partial z^{\alpha}}(z)\Big| < \infty$$
        hold for all $j\in\{1,\dots,q\},$  where $m_1, \dots, m_q$ are positive integers and may be $\infty,$ 
with $\sum_{j=1}^{q}\dfrac{d_j}{m_j}<\dfrac{(q-n_{\mathcal D}-1)m_{\mathcal D}}{n_{\mathcal D}}.$ Then the holomorphic mapping 
$f$  from $\mathbb D-S$ into $\mathbb P^n(\mathbb C)$ extends to a holomorphic mapping from $\mathbb D$ into $\mathbb P^n(\mathbb C).$
\end{theorem}

\begin{theorem}\label{th12a}
Let $S$ be an analytic subset of a domain $\mathbb D$ in $\mathbb C^m$ with codimension one, whose singularities are normal crossings. Let $f$ be
 a holomorphic mapping from $\mathbb D-S$ into $\mathbb P^n(\mathbb C).$  Let $D_1, \dots, D_q$ be $q$ $(\ge 2n_{\mathcal D}+1)$ moving hypersurfaces
 in $\mathbb P^n(\mathbb C).$ Let $Q_j^{*}=\sum_{i=0}^{n_{\mathcal D}}d_{ij}\omega_i,\; (\omega_i=x^{I_i})$ in 
$H_{\mathbb D}[\omega_0,\dots,\omega_{n_{\mathcal D}}]$ be moving linear form
 defines the associated hyperplane $H_j^{*}$  of $D_j$ such that for any $z\in \mathbb D:$
 $$D(Q_1^{*}, \dots, Q_q^{*})(z) > 0.$$ 
Let $\tilde f=(f_0,\dots, f_n)$ be a reduced representation of $f$  on $\mathbb D-S$ 
with $ f_i(z)\ne 0$ for some $i$ and $z,$ and when $m_j\geq 2$
        $$\sup_{1\leq |\alpha|\leq m_j-1, z \in f^{-1}(D_{j})}\Big|\dfrac{\partial^{\alpha}(Q_{j}(\tilde f_{i}))}{\partial z^{\alpha}}(z)\Big| < \infty$$
        hold for all $j\in\{1,\dots,q\},$  where $m_1, \dots, m_q$ are positive integers and may be $\infty,$ 
with $\sum_{j=1}^{q}\dfrac{d_j}{m_j}<\dfrac{(q-n_{\mathcal D}-1)m_{\mathcal D}}{n_{\mathcal D}}.$  Then the holomorphic 
mapping $f$  from $\mathbb D-S$ into $\mathbb P^n(\mathbb C)$ extends to a holomorphic mapping from $\mathbb D$ into $\mathbb P^n(\mathbb C).$
\end{theorem}
\begin{theorem}\label{th13n}
Let $\mathcal F$ be a family of holomorphic mappings of a domain $\mathbb D$ in $\mathbb C^m$ into $\mathbb P^n(\mathbb C).$ Let $q\ge 2n+1$ be 
a positive integer. Let $m_1, \dots, m_q$ be positive integers or $\infty$ such that
$$ \sum_{j=1}^{q}\dfrac{1}{m_j}<\dfrac{q-n-1}{n}.$$
Suppose that for each $f\in \mathcal F,$ there exist $n+1$ moving hypersurfaces $T_{0,f},$ $ \dots, T_{n,f}$  in $\mathbb P^n(\mathbb C)$ of common degree 
and that there exist $q$ moving hypersurfaces $D_{1,f},$ $ \dots, D_{q,f}$ in $\overset{\sim}S(\{T_{i,f}\}_{i=0}^{n})$ such that the following 
conditions are satisfied  on each compact subset $K\subset \mathbb D:$

\noindent $(i)$ For each $0\le i \le n,$ the coefficients of the homogeneous polynomials $P_{i,f}$ which define the $T_{i,f}$ are bounded above 
uniformly for all $f$ in $\mathcal F$ on compact subset $K$ of $\mathbb D$ for all $1\le j\le q,$ the coefficients $b_{ij}(f)$ of the linear combinations of the
 $P_{i,f}, i = 0, \dots, n,$ which define the homogeneous polynomials $Q_{j,f}$  are bounded above uniformly for
all $f$ in $\mathcal F$ on compact subset $K$ of $\mathbb D,$ where $Q_{j,f} \in S(\{P_{i,f}\}_{i=0}^{n})$ is a homogeneous polynomial defined the $D_{j,f},$ 
and for any fixed $z\in \mathbb D,$
$$ \inf_{f\in F}D(Q_{1,f}, \dots, Q_{q,f})(z) > 0.$$

\noindent $(ii)$ Let $\tilde f=(f_0,\dots, f_n)$ be a reduced representation of $f\in \mathcal{F}$  on $\mathbb D$ 
with $ f_i(z)\ne 0$ for some $i$ and $z,$ and when $m_j\geq 2$
        $$\sup_{1\leq |\alpha|\leq m_j-1, z \in f^{-1}(D_{j,f})\cap K}\Big|\dfrac{\partial^{\alpha}(Q_{j,f}(\tilde f_{i}))}{\partial z^{\alpha}}(z)\Big| < \infty$$
        hold for all $j\in\{1,\dots,q\}.$  

Then $\mathcal F$ is a holomorphically normal family on $\mathbb D$.
\end{theorem}
\begin{theorem}\label{th13}
Let $S$ be an analytic subset of a domain $\mathbb D$ in $\mathbb C^m$ with $\dim_{\mathbb C} S\le m-2.$ Let $f$ be a holomorphic mapping from $\mathbb D-S$ into $\mathbb P^n(\mathbb C).$ Suppose that there exist $n+1$ moving hypersurfaces $T_0, \dots, T_n$  in $\mathbb P^n(\mathbb C)$ of common degree which are defined by homogeneous polynomial $P_0,\dots,P_n\in H_{\mathbb D}[x_0,\dots,x_n]$ respectively, and that there exist $q$ moving hypersurfaces $D_1, \dots, D_q$ in $\overset{\sim}S(\{T_i\}_{i=0}^{n})$ such that the following conditions are satisfied:

\noindent $(i)$ $Q_j \in S(\{P_i\}_{i=0}^{n})$ is a homogeneous polynomial defined the $D_j,$ and for any $z\in \mathbb D,$
$$D(Q_1, \dots, Q_q)(z) > 0.$$

\noindent $(ii)$ Let $\tilde f=(f_0,\dots, f_n)$ be a reduced representation of $f$  on $\mathbb D-S$ 
with $ f_i(z)\ne 0$ for some $i$ and $z,$ and when $m_j\geq 2$
        $$\sup_{1\leq |\alpha|\leq m_j-1, z \in f^{-1}(D_{j})}\Big|\dfrac{\partial^{\alpha}(Q_{j}(\tilde f_{i}))}{\partial z^{\alpha}}(z)\Big| < \infty$$
        hold for all $j\in\{1,\dots,q\},$  where $m_1, \dots, m_q$ are positive integers and may be $\infty,$ 
with $$ \sum_{j=1}^{q}\dfrac{1}{m_j}<\dfrac{q-n-1}{n}.$$ Then the holomorphic mapping $f$  from $\mathbb D-S$ into $\mathbb P^n(\mathbb C)$ extends to a holomorphic mapping from $\mathbb D$ into $\mathbb P^n(\mathbb C).$
\end{theorem}

\begin{theorem}\label{th13a}
Let $S$ be an analytic subset of a domain $\mathbb D$ in $\mathbb C^m$ with codimension one, whose singularities are normal crossings. Let $f$ be a holomorphic mapping from $\mathbb D-S$ into $\mathbb P^n(\mathbb C).$ Suppose that there exist $n+1$ moving hypersurfaces $T_0, \dots, T_n$  in $\mathbb P^n(\mathbb C)$ of common degree which are defined by homogeneous polynomials $P_0,\dots,P_n\in H_{\mathbb D}[x_0,\dots,x_n]$ respectively, and that there exist $q$ moving hypersurfaces $D_1, \dots, D_q$ in $\overset{\sim}S(\{T_i\}_{i=0}^{n})$ such that the following conditions are satisfied:

\noindent $(i)$ $Q_j \in S(\{P_i\}_{i=0}^{n})$ is a homogeneous polynomial defined the $D_j,$ and for any $z\in \mathbb D,$
$$D(Q_1, \dots, Q_q)(z) > 0.$$

\noindent $(ii)$ Let $\tilde f=(f_0,\dots, f_n)$ be a reduced representation of $f$  on $\mathbb D-S$ 
with $ f_i(z)\ne 0$ for some $i$ and $z,$ and when $m_j\geq 2$
        $$\sup_{1\leq |\alpha|\leq m_j-1, z \in f^{-1}(D_{j})}\Big|\dfrac{\partial^{\alpha}(Q_{j}(\tilde f_{i}))}{\partial z^{\alpha}}(z)\Big| < \infty$$
        hold for all $j\in\{1,\dots,q\},$  where $m_1, \dots, m_q$ are positive integers and may be $\infty,$ 
with $$ \sum_{j=1}^{q}\dfrac{1}{m_j}<\dfrac{q-n-1}{n}.$$
 Then the holomorphic mapping $f$  from $\mathbb D-S$ into $\mathbb P^n(\mathbb C)$ extends to a holomorphic mapping from $\mathbb D$ into $\mathbb P^n(\mathbb C).$
\end{theorem}
\begin{remark}\label{rm2}
\noindent $(i)$ Our results cover the results due to Tu-Li \cite{TL, TL1}.

\noindent $(ii)$ Theorem \ref{th13n} is a extension the result due to Dethloff-Thai-Trang (Theorem \ref{th8}).

\end{remark}
Finaly, we consider the case moving hypersurfaces in weakly general position in $\mathbb P^{n}(\mathbb C).$ We prove the results as follows:

\begin{theorem}\label{th14}
Let $\mathcal F$ be a family of holomorphic mappings of a domain $\mathbb D$ in $\mathbb C^m$ into $\mathbb P^n(\mathbb C).$ Let $q\ge n n_{\mathcal D}+2n+1$ be a positive integer. Let $m_1, \dots, m_q$ be positive integers or $\infty$ such that
$$\sum_{j=1}^{q}\dfrac{1}{m_j}<\dfrac{q-n-(n-1)(n_{\mathcal D}+1)}{n_{\mathcal D}(n_{\mathcal D}+2)}.$$
Suppose that for each $f\in \mathcal F,$ there exist $q$ moving hypersurfaces $D_{1,f}, \dots, D_{q,f}$ which are defined by homogeneous polynomials
 $Q_{1,f},\dots,Q_{q,f}\in H_{\mathbb D}[x_0, \dots, x_n]$ with degree $d_1,\dots,d_q$ respectively, such that the following conditions are satisfied 
  on each compact subset $K\subset \mathbb D:$

\noindent $(i)$ For each $1\le j \le q,$ the coefficients of the homogeneous polynomials $Q_{j,f}$ are bounded above uniformly for all $f$ in $\mathcal F$ 
on compact subset $K$ of $\mathbb D,$ and for any $\{f_p\}\subset \mathcal F,$ and for any fixed $z\in \mathbb D,$
 $$ \inf_{p\in \mathbb N}D(Q_{1,f_p}, \dots, Q_{q,f_p})(z) > 0.$$

\noindent $(ii)$ Assume that for each $f\in \mathcal F,$ and $\tilde f=(f_0,\dots, f_n)$ is a reduced representation of $f\in \mathcal{F}$  on $\mathbb D$ 
with $ f_i(z)\ne 0$ for 
some $i$ and $z,$ and when $m_j\geq 2$
        $$\sup_{1\leq |\alpha|\leq m_j-1, z \in f^{-1}(D_{j,f})\cap K}\Big|\dfrac{\partial^{\alpha}(Q_{j,f}(\tilde f_{i}))}{\partial z^{\alpha}}(z)\Big| < \infty$$
        hold for all $j\in\{1,\dots,q\}$ and all $f\in \mathcal F.$

Then $\mathcal F$ is a holomorphically normal family on $\mathbb D$.
\end{theorem}
From Theorem \ref{th14}, we get the result as follows:

\begin{cor}\label{corth14}
Let $\mathcal F$ be a family of holomorphic mappings of a domain $\mathbb D$ in $\mathbb C^m$ into $\mathbb P^n(\mathbb C).$ Let 
$q\ge n n_{\mathcal D}+2n+1$ be a positive integer. Let $m_1, \dots, m_q$ be positive integers or $\infty$ such that
$$\sum_{j=1}^{q}\dfrac{1}{m_j}<\dfrac{q-n-(n-1)(n_{\mathcal D}+1)}{n_{\mathcal D}(n_{\mathcal D}+2)}.$$
Suppose that for each $f\in \mathcal F,$ there exist $q$ moving hypersurfaces $D_{1,f}, \dots, D_{q,f}$ which are defined by homogeneous polynomials
 $Q_{1,f},\dots,Q_{q,f}\in H_{\mathbb D}[x_0, \dots, x_n]$ with degree $d_1,\dots,d_q$ respectively, such that the following conditions are satisfied
  on each compact subset $K\subset \mathbb D:$

\noindent $(i)$ For each $1\le j \le q,$ the coefficients of the homogeneous polynomials $Q_{j,f}$ are bounded above uniformly for all $f$ in $\mathcal F$ 
on compact subset $K$ of $\mathbb D,$ and for any $\{f_p\}\subset \mathcal F,$ and for any fixed $z\in \mathbb D,$
 $$ \inf_{p\in \mathbb N}D(Q_{1,f_p}, \dots, Q_{q,f_p})(z) > 0.$$

\noindent $(ii)$ Assume that $f$ intersects $D_{j,f}$ with multiplicity at least $m_j$ for each $1\le j\le q,$ for all $f\in \mathcal F.$

Then $\mathcal F$ is a holomorphically normal family on $\mathbb D$.
\end{cor}
Let $\Lambda^l(S)$ denote the real $l$-dimensional Hausdorff measure of $S\subset \mathbb C^m.$  We have the result as follows:
\begin{theorem}\label{th15}
Let $\mathcal F$ be a family of meromorphic mappings of a domain $\mathbb D$ in $\mathbb C^m$ into $\mathbb P^n(\mathbb C).$ Let
 $q\ge n n_{\mathcal D}+2n+1$
 be a positive integer. Suppose that for each $f\in \mathcal F,$ there exist $q$ moving hypersurfaces $D_{1,f}, \dots, D_{q,f}$ which are defined by 
homogeneous 
polynomials $Q_{1,f},\dots,Q_{q,f}\in H_{\mathbb D}[x_0, \dots, x_n]$ with degree $d_1,\dots,d_q$ respectively, such that the following conditions are 
satisfied:

\noindent $(i)$ For each $1\le j \le q,$ the coefficients of the homogeneous polynomials $Q_{j,f}$ are bounded above uniformly for all $f$ in $\mathcal F$ 
on compact
 subsets of $\mathbb D,$ and for any $\{f_p\}\subset \mathcal F,$ there exists $z\in \mathbb D$ (which may depend on sequence) such that
 $$ \inf_{p\in \mathbb N}D(Q_{1,f_p}, \dots, Q_{q,f_p})(z) > 0.$$

\noindent $(ii)$ For any fixed compact $K$ of $\mathbb D,$ the $2(m-1)$-dimensional Lebesgue areas of $f^{-1}(D_{k,f})\cap K$ $(1\le k\le n+1)$ counting
 multiplicities are bounded above for all $f\in \mathcal F$ (in particular, $f(\mathbb D)\not\subset D_{k,f} (1\le k\le n+1)).$

\noindent $(iii)$ There exists a closed subset  $S$ in $\mathbb D$ with $\Lambda^{2m-1}(S)=0$ such that for any fixed compact subset $K$ of 
$\mathbb D-S,$ the $2(m-1)$-dimensional Lebesgue areas of
$$  \{z\in \text{supp} \nu_{Q_{k,f}(\tilde f)}: \nu_{Q_{k,f}(\tilde f)}(z)<m_k\}\cap K \; (n+2\le k\le q)$$
ignoring multiplicities for all $f\in \mathcal F$ are bounded above, where $\{m_k\}_{k=n+2}^{q}$ are fixed positive integers or $\infty$ with
$$ \sum_{k=n+2}^{q}\dfrac{1}{m_k}<\dfrac{q-n-(n-1)(n_{\mathcal D}+1)}{n_{\mathcal D}(n_{\mathcal D}+2)}.$$

Then $\mathcal F$ is a meromorphicaly normal family on $\mathbb D.$
\end{theorem}

\begin{theorem}\label{th16}
Let $S$ be an analytic subset of a domain $\mathbb D$ in $\mathbb C^m$ with $\dim_{\mathbb C} S\le m-2.$ Let $f$ be a holomorphic mapping from
 $\mathbb D-S$ into $\mathbb P^n(\mathbb C).$ Let $D_1, \dots, D_q$ be $q$ $(\ge  nn_{\mathcal D}+2n+1)$ moving hypersurfaces in 
$\mathbb P^n(\mathbb C)$ which are defined by homogeneous polynomials $Q_1,\dots, Q_q\in H_{\mathbb D}[x_0, \dots, x_n]$ with degree 
$d_1,\dots,d_q$ respectively such that such that for any $z\in \mathbb D:$
 $$D(Q_1, \dots, Q_q)(z) > 0.$$ Let $\tilde f=(f_0,\dots, f_n)$ be a reduced representation of $f$  on $\mathbb D-S$ 
with $ f_i(z)\ne 0$ for some $i$ and $z,$ and when $m_j\geq 2$
        $$\sup_{1\leq |\alpha|\leq m_j-1, z \in f^{-1}(D_{j})}\Big|\dfrac{\partial^{\alpha}(Q_{j}(\tilde f_{i}))}{\partial z^{\alpha}}(z)\Big| < \infty$$
        hold for all $j\in\{1,\dots,q\},$  where $m_1, \dots, m_q$ are positive integers and may be $\infty,$  with 
$$\sum_{j=1}^{q}\dfrac{1}{m_j}<\dfrac{q-n-(n-1)(n_{\mathcal D}+1)}{n_{\mathcal D}(n_{\mathcal D}+2)}.$$ Then the holomorphic
 mapping $f$  from $\mathbb D-S$ into $\mathbb P^n(\mathbb C)$ extends to a holomorphic mapping from $\mathbb D$ into $\mathbb P^n(\mathbb C).$
\end{theorem}

\begin{theorem}\label{th17}
Let $S$ be an analytic subset of a domain $\mathbb D$ in $\mathbb C^m$ with codimension one, whose singularities are normal crossings. Let $f$ be a
 holomorphic mapping from $\mathbb D-S$ into $\mathbb P^n(\mathbb C).$  Let $D_1, \dots, D_q$ be $q$ $(\ge  n n_{\mathcal D}+2n+1)$ moving 
hypersurfaces in $\mathbb P^n(\mathbb C)$ which are defined by homogeneous polynomials $Q_1,\dots, Q_q\in H_{\mathbb D}[x_0, \dots, x_n]$ with 
degree $d_1,\dots,d_q$ respectively such that for any $z\in \mathbb D:$
 $$D(Q_1, \dots, Q_q)(z) > 0.$$ 
Let $\tilde f=(f_0,\dots, f_n)$ be a reduced representation of $f$  on $\mathbb D-S$ 
with $ f_i(z)\ne 0$ for some $i$ and $z,$ and when $m_j\geq 2$
        $$\sup_{1\leq |\alpha|\leq m_j-1, z \in f^{-1}(D_{j})}\Big|\dfrac{\partial^{\alpha}(Q_{j}(\tilde f_{i}))}{\partial z^{\alpha}}(z)\Big| < \infty$$
        hold for all $j\in\{1,\dots,q\},$  where $m_1, \dots, m_q$ are positive integers and may be $\infty,$ with 
$$\sum_{j=1}^{q}\dfrac{1}{m_j}<\dfrac{q-n-(n-1)(n_{\mathcal D}+1)}{n_{\mathcal D}(n_{\mathcal D}+2)}.$$ Then the holomorphic 
mapping $f$  from $\mathbb D-S$ into $\mathbb P^n(\mathbb C)$ extends to a holomorphic mapping from $\mathbb D$ into $\mathbb P^n(\mathbb C).$
\end{theorem}

\section{Some lemmas}
In order to prove our results, we need some lemmas as follows:
\begin{defi}\label{dn6}\cite{ST}
Let $\{A_i\}$ be a sequence of closed subsets of $\mathbb D.$ It is said to converge to a closed subset $A$ of $\mathbb D$ if and only if $A$ coincides 
with the set of all $z$ such that every neighborhood $U$ of $z$ intersects $A_i$ for all but finitely many $i$ and, simultaneously, with the set of all $z$ such
 that every $U$ intersects $A_i$ for infinitely many $i.$
\end{defi}

\begin{lemma}\label{lmb2}\cite{ST}
Let $\{N_i\}$ be a sequence of pure $(m - 1)$-dimensional analytic subsets of a domain $\mathbb D$ in $\mathbb C^m.$ If the $2(m-1)$-
 dimensional Lebesgue areas of $N_i\cap K$ ignoring multiplicities $(i = 1, 2, \dots)$ are bounded above for any fixed compact subset $K$ of $\mathbb D,$ 
then $\{N_i\}$ is normal in the sense of the convergence of closed subsets in $\mathbb D.$
\end{lemma}
\begin{lemma}\label{lmb3}\cite{ST}
Let $\{N_i\}$ be a sequence of pure $(m-1)$-dimensional analytic subsets of a domain $\mathbb D$ in $\mathbb C^m.$ Assume that
the $2(m - 1)$-dimensional Lebesgue areas of $N_i\cap K$ ignoring multiplicities $(i = 1, 2, \dots)$ are bounded above for any fixed compact subset 
$K$ of $\mathbb D$ and that $\{N_i\}$ converges to $N$ as a sequence of closed subsets of $\mathbb D.$ Then $N$ is
either empty or a pure $(m-1)$-dimensional analytic subset of $\mathbb D$.
\end{lemma}

\begin{lemma}\label{lm5c}\cite{GK}
Let $F$ be a family of holomorphic mappings of a domain $\mathbb D$ in $\mathbb C^m$ into $\mathbb P^{n}(\mathbb C).$ Then the family $F$ is 
not normal on $\mathbb D$ if and only if there exist a compact subset $K_0\subset \mathbb D$ and sequences $\{f_p\}\subset F,$ $\{y_p\}\subset K_0,$ 
$ \{r_p\} \subset \mathbb R$ with $r_p > 0$ and $r_p\to 0^{+},$ and $\{u_p\}\subset \mathbb C^m$ which are unit vectors such that
$$g_p(\xi) := f_p(y_p + r_pu_p\xi),$$
where $\xi \in \mathbb C$ such that $y_p +r_pu_p\xi \in \mathbb D,$ converges uniformly on compact subsets of $\mathbb C$ to a
 nonconstant holomorphic map $g$ of $\mathbb C$ to $\mathbb P^{n}(\mathbb C).$
\end{lemma}
\begin{defi}\label{normal}
Let $\Omega\subset \mathbb C^m$ be a hyperbolic domain and let $M$ be a complete complex Hermittian manifold with metric $ds_M^{2}.$ A holomorphic
 mapping $f(z)$ from $\Omega$ into $M$ is said to be a normal holomorphic mapping from $\Omega$ into $M$ if and only if there exists a positive 
constant $c$ such that for all $z\in \Omega$ and all $\xi\in T_z(\Omega),$
$$ |ds_M^2(f(z), df(z)(\xi))|\le cK_{\Omega}(z,\xi),$$
where $df(z)$ is the mapping from $T_z(\Omega)$ into $T_{f(z)}(M)$ induced by $f$ and $K_{\Omega}$ denote the infinitesimal Kobayashi metric 
on $\Omega.$
\end{defi}
\begin{lemma}\label{nf}\cite{TL1}
Let $f$ be a holomorphic mapping of a hyperbolic domain $\mathbb D$ in $\mathbb C^m$ into $\mathbb P^{n}(\mathbb C).$ Then $f$ is not normal on
 $\mathbb D$ if and only if there exist $\{y_p\}\subset \mathbb D,$ $\{r_p\}$ with $r_p>0$ and $r_p\to 0^{+}$ and $\{u_p\}\subset\mathbb C^m$ Euclidean unit vectors such that
$$ g_p(\xi):=f(y_p+r_pu_p\xi), \xi \in \mathbb C, $$
where $\lim_{p\to\infty}\dfrac{r_p}{d(y_p, \mathbb C^m\setminus \mathbb D)}=0$ (where $d(p,q)$ is the Euclidean distance between $p$ and 
$q$ in $\mathbb C^m$), converges uniformly on compact subsets of $\mathbb C$ to a non-constant holomorphic mapping $g$ of $\mathbb C$ 
to $\mathbb P^{n}(\mathbb C).$
\end{lemma}
\begin{lemma}\label{lm5a}\cite{DTT}
Let natural numbers $n$ and $q\ge n + 1$ be fixed. Let $D_{kp}\; (1\le k\le q, p\ge 1)$ be moving hypersurfaces in $\mathbb P^{n}(\mathbb C)$ such that 
the following conditions are satisfied:

\noindent $(i)$ For each $1\le k\le q, p\ge 1,$ the coefficients of the homogeneous polynomials $Q_{kp}$ which define the $D_{kp}$ are bounded above
 uniformly for all $p\ge 1$ on compact subsets of $\mathbb D.$

\noindent $(ii)$ There exists $z_0\in \mathbb D$ such that
$$ \inf_{p\in \mathbb N}\{D(Q_{1p}, \dots, Q_{qp})(z_0)\}>\delta>0$$
Then, we have the following:

\noindent $(a)$ There exists a subsequence $\{j_p\}\subset \mathbb N$ such that for $1\le k \le q,$ $Q_{kj_p}$ converge uniformly on compact subsets 
of $\mathbb D$ to not identically zero homogeneous polynomials $Q_k$ (meaning that the $Q_{kj_p}$ and $Q_k$ are homogeneous polynomials in
 $H_{\mathbb D}[x_0, \dots, x_N]$ of the same degree, and all their coefficients converge uniformly on compact subsets of $\mathbb D$). Moreover, 
we have that $D(Q_1, \dots, Q_q)(z_0) >\delta> 0,$ the hypersurfaces $Q_1(z_0), \dots, Q_q(z_0)$ are located in general position, and the moving 
hypersurfaces $Q_1(z), \dots, Q_q(z)$ are located in (weakly) general position.

\noindent $(b)$ There exist a subsequence $\{j_p\} \subset \mathbb N$ and $r = r(\delta) > 0$ such that
$$ \inf_{p\in \mathbb N}D(Q_{1j_p}, \dots, Q_{qj_p})(z)>\dfrac{\delta}{4}, \forall z\in B(z_0, r).$$
\end{lemma}

\begin{lemma}\label{lm5b}\cite{DTT}
Let $\{f_p\}$ be a sequence of meromorphic mappings of a domain $\mathbb D$ in $\mathbb C^m$ into $\mathbb P^n(\mathbb C),$ and let $S$ be a closed
 subset of $\mathbb D$ with $\Lambda^{2m-1}(S) = 0.$ Suppose that $\{f_p\}$ meromorphically converges on $\mathbb D-S$ to a meromorphic mapping $f$ 
of  $\mathbb D-S$ into $\mathbb P^n(\mathbb C).$ Suppose that, for each $f_p,$ there exist $n+1$ moving hypersurfaces $D_{1,f_p}, \dots, D_{n+1,f_p}$ in
 $\mathbb P^n(\mathbb C),$ where the moving hypersurfaces $D_{i,f_p}$ may depend on $f_p$, such that the following three conditions are satisfied:

\noindent $(i)$ For each $1\le k \le n+1,$ the coefficients of homogeneous polynomial $Q_{k,f_p}$ which define $D_{k,f_p}$ for all $f_p$ are bounded above 
uniformly for all $p\ge 1$ on compact subsets of $\mathbb D.$

\noindent $(ii)$ There exists $z_0\in \mathbb D$ such that
 $$\inf_{p\ge 1}D(Q_{1,f_p}, \dots, Q_{n+1,f_p}(z_0))>0;$$
 
\noindent $(iii)$ The $2(m-1)$-dimensional Lebesgue areas of $(f_p)^{-1}(D_{k,f_p})\cap K$ $(1\le k\le n+1, p\ge 1)$ counting multiplicities are bounded above 
for any fixed compact subset $K$ of $\mathbb D.$

Then we have the following:

\noindent $(a)$ $\{f_p\}$ has a meromorphically convergent subsequence on $\mathbb D$, and

\noindent $(b)$ if, moreover, $\{f_p\}$ is a sequence of holomorphic mappings of a domain $\mathbb D$ in $\mathbb C^m$ into $\mathbb P^{n}(\mathbb C)$ 
and condition $(iii)$ is sharpened to
$$f_p(\mathbb D)\cap D_{k,f_p} = \emptyset\; (1\le k\le n+1, p\ge 1),$$
then $\{f_p\}$ has a subsequence which converges uniformly on compact subsets of $\mathbb D$ to a holomorphic mapping of $\mathbb D$ to
 $\mathbb P^n(\mathbb C).$
\end{lemma}

The following lemma is Picard type theorem in Nevanlinna-Cartan theory due to Nochka:
\begin{lemma}\label{lm5}\cite{EN}
Suppose that $q\ge 2n + 1$ hyperplanes $H_1, \dots, H_q$ are given in general position in $\mathbb P^{n}(\mathbb C)$ and that $q$ positive integers (may be $\infty$) $m_1,\dots, m_q$ are given such that
$$\sum_{j=1}^{q}\dfrac{1}{m_j}<\dfrac{q-n-1}{n}.$$
Then there does not exist a nonconstant holomorphic mapping $f: \mathbb C \to \mathbb P^{n}(\mathbb C)$ such that $f$ intersects $H_j$ with multiplicity at least $m_j\; (1\le j\le q).$
\end{lemma}
\begin{lemma}\label{lm6}\cite{DTT}
Let $f = (f_0: \dots: f_n): \mathbb C \to \mathbb P^{n}(\mathbb C)$ be a holomorphic mapping, and let $\{P_i\}_{i=0}^{n+1}$ be 
$n+1$ homogeneous polynomials in general position of common degree in $\mathbb C[x_0, \dots, x_n].$ Assume that
$$F = (F_0: \dots: F_n) : \mathbb P^{n}(\mathbb C) \to  \mathbb P^n(\mathbb C)$$
is the mapping defined by
$$F_i({\bf x}) = P_i({\bf x})\;  (0 \le i \le n).$$
Then, $F(f)$ is a constant map if and only if $f$ is a constant map.
\end{lemma}
\begin{lemma}\label{lm7}\cite{DTT}
Let $P_0, \dots, P_n$ be $n+1$ homogeneous polynomials of common degree in $\mathbb C[x_0, \dots, x_n].$ Also let $\{Q_j\}_{j=1}^{q}\; (q\ge n+1)$ be homogeneous polynomials in $S(\{P_i\}_{i=0}^{n})$ such that
\begin{align*}
D(Q_1, \dots, Q_q)(z)=\prod_{1\le j_0<\dots<j_n\le q}\inf_{||{\bf x}||=1}(|Q_{j_0}(z)({\bf x})|^2+\dots+|Q_{j_n}(z)({\bf x})|^2)>0,
\end{align*}
where $Q_j(z)({\bf x}) = \sum_{|I|=d_j}a_{jI}(z){\bf x}^{I}$ and $||{\bf x}||=(\sum_{j=0}^{n}|x_j|^{2})^{1/2}.$

Then $\{Q_j\}_{j=1}^{q}$ are located in general position in $S(\{P_i\}_{i=0}^{n})$ and $\{P_i\}_{i=0}^{n}$ are located in general position in $\mathbb P^{n}(\mathbb C).$
\end{lemma}

\begin{lemma}\label{lm9}\cite{HT}
Let $f$ be a meromorphic nonconstant mappings from $\mathbb C^m$ into $\mathbb P^n(\mathbb C).$ Let $D_i\; (i=1, \dots, q), q\ge nn_{\mathcal D}+n+1$ 
be slowly (with respective $f$) moving hypersurfaces of $\mathbb P^n(\mathbb C)$ in weakly general position which is defined by homogeneous polynomials $Q_i\in H_{\mathbb D}[x_0, \dots, x_n]$ with $\deg Q_i=d_i\; (i=1,\dots,q).$ Set $m_{\mathcal D}=lcm(d_1, \dots, d_q)$ and $n_{\mathcal D}=\binom{n+m_{\mathcal D}}{n}-1.$ Let $m_1, \dots, m_q$ be positive integers or $\infty$ such that
$$\sum_{j=1}^{q}\dfrac{1}{m_j}<\dfrac{q-(n-1)(n_{\mathcal D}+1)}{n_{\mathcal D}(n_{\mathcal D}+2)}.$$
Assume that $Q_j(\tilde f)\not\equiv 0\; (1\le j\le q)$  and  $f$ intersects $D_j$ with multiplicity at least $m_j$ for each $1\le j\le q,$ then $f$ must be 
constant mapping.
\end{lemma}

\section{Proof of Theorems}
\def\theequation{3.\arabic{equation}}
\setcounter{equation}{0}
\begin{proof}[Proof of Theorem \ref{th10}]
Without loss of generality, we may assume that $\mathbb D$ is the unit disc. Suppose that $F$ is not normal on $\mathbb D.$ Then, by Lemma \ref{lm5c}, there exists a subsequence denoted by $\{f_p\}\subset F$ and 
$y_0\in K_0$, $\{y_p\}_{p=1}^{\infty}\subset K_0$ with $y_p\to y_0,$ $\{r_p\}\subset (0, +\infty)$ with $r_p\to 0^{+},$ and 
$\{u_p\}\subset \mathbb C^m$, which are unit vectors, such that $g_p(\xi) := f_p(y_p + r_pu_p\xi)$ converges uniformly on compact subsets of $\mathbb C$
 to a nonconstant holomorphic map $g$ of $\mathbb C$ into $\mathbb P^{n}(\mathbb C).$  Let 
$$\varrho_{m_{\mathcal D}}: \mathbb P^n(\mathbb C) \to \mathbb P^{n_{\mathcal D}}(\mathbb C)$$ be the {\it Veronese embedding of degree
 $m_{\mathcal D}$} which is given by
$$\varrho_{m_{\mathcal D}}(x) = (w_0({\bf x}):\dots:w_{n_{\mathcal D}}({\bf x})), \ \text{ where } \ w_j({\bf x}) = {\bf x}^{I_j}, \ j=0,\dots,n_{\mathcal D}.$$
For $w=(w_0:\dots:w_{n_{\mathcal D}}),$ we denote ${\bf w}=(w_0,\dots,w_{n_{\mathcal D}}).$ We see that $G_p:=\rho_{m_{\mathcal D}}(g_p)$ converges uniformly on compact subsets of $\mathbb C$ to $ G:=\rho_{m_{\mathcal D}}(g).$ By Lemma \ref{lm5a}, $Q_{jp}^{*}=Q_{jp}^{{m_{\mathcal D}}/{d_j}}:=Q_{j,f_p}^{{m_{\mathcal D}}/{d_j}}$ converge uniformaly on compact subset of $\mathbb D$ to $Q_j^{*}, 1\le j\le q,$ and 
$$ D(Q_1^{*}, \dots, Q_q^{*})(z)>\delta(z)>0$$
for any fixed $z\in \mathbb D.$ In particular, the moving hyperplanes $H_1^{*}, \dots, H_q^{*}$ defining by moving linear forms 
$Q_1^{*}, \dots, Q_q^{*}$ (respectively) are located in pointwise general position in $\mathbb P^{n_{\mathcal D}}(\mathbb C).$ 

We see that $y_p + r_pu_p\xi$ belong to a closed disc $K=\left\{z\in \mathbb{C}^m: ||z||\le \dfrac{1+||y_0||}{2}\right\}$ of $\mathbb D$ for all $\xi$ in a subset 
compact of $\mathbb C$ if $p$ is large enough. Thus for $j\in\{1,...,q\}$, there exist reduced representations 
$\tilde{f}_p=(f_{0p},...,f_{np}), \tilde g=(g_0, \dots , g_n)$ of $f_p,g$ respectively such that as $p\to\infty$, $\tilde{g}_p(\xi)= \tilde{f}_p(y_p + r_pu_p\xi)
\to \tilde{g}(\xi) $ and hence
\begin{align}\label{m2}
 Q_{jp,y_p + r_pu_p\xi}(\tilde{g}_p(\xi) )\to Q_{j,y_0}(\tilde g(\xi))
\end{align}
converge uniformly on compact subsets of $\mathbb{C}$.

\noindent{\bf Claim 1}. $g(\mathbb{C})$ intersects with $D_{j,y_0}$ with multiplicity at least $m_j.$

Indeed, for $\xi_0\in \mathbb C$ with $Q_{j,y_0}(\tilde g(\xi_0))=0$, there exist a disc $B(\xi_0, r_0)$ in $\mathbb C$ and an index $i\in \{0,\dots, n\}$
 satisfying $
    g(B(\xi_0, r_0))\subset \{[x_0:...:x_n]\in \mathbb P^n(\mathbb C) : x_i\ne 0\}.$
Since $g_i(\xi)\ne 0$ on $B(\xi_0, r_0)$,  Hurwitz's theorem shows that $g_{ip}(\xi)=f_{ip}(y_p + r_pu_p\xi)\ne 0$ on  $B(\xi_0, r_0)$ when $p$ is large
 enough. Thus $\tilde{g}_{ip}(\xi)=\left(\dfrac{g_{0p}}{g_{ip}},\dots, \dfrac{g_{np}}{g_{ip}}\right)\to \tilde g_i(\xi)$, and hence 
$Q_{jp,y_p + r_pu_p\xi}(\tilde g_{ip}(\xi))\to Q_{j,y_0}(\tilde g_i(\xi))$ converge uniformly on $B(\xi_0, r_0)$, which further implies that
\begin{align}\label{m3}
(Q_{jp,y_p + r_pu_p\xi}(\tilde g_{ip}(\xi)))^{(k)}\to (Q_{j,y_0}(\tilde g_i(\xi)))^{(k)}
\end{align}
converge uniformly on $B(\xi_0, r_0)$ for all $k\geq 1$.

By the condition (ii) in Theorem \ref{th10}, when $m_j\geq 2,$
$$\sup_{1\leq |\alpha|\leq m_j-1, z \in f_p^{-1}(D_{jp})\cap K}|\dfrac{\partial^{\alpha}(Q_{jp}(\tilde f_{ip}))}{\partial z^{\alpha}}(z)| < \infty$$
hold for all $j\in\{1,\dots,q\}$ and all $p\geq 1,$ where $f_{ip}(z)\ne 0$  and $D_{jp}$ is the moving hypersurface associated to $Q_{jp}$. Since 
$Q_{j,y_0}(\tilde g(\xi_0))=0$, then Hurwitz's theorem shows that there exits $\xi_p\to \xi_0$ such that $ Q_{jp,y_p + r_pu_p\xi_p}(\tilde{g}_p(\xi_p) )=0,$ 
which implies $y_p+ r_pu_p\xi_p\in f_{p}^{-1}(D_{jp}).$ Hence there exists $M>0$ such that when $m_j\geq 2\ (1\leq j\leq q)$ and $p$ is large enough
\begin{align*}
 |\dfrac{\partial^{\alpha}(Q_{jp}(\tilde f_{ip}))}{\partial z^{\alpha}}(y_p + r_pu_p\xi_p)|\le M
\end{align*}
for all $\alpha=(\alpha_1,\dots,\alpha_m): 1\le |\alpha| \le m_j-1$ and some $i\in \{0,\dots,n\},$ which means 
\begin{align}\label{m5}
\left |(Q_{jp,y_p + r_pu_p\xi}(\tilde g_{ip}(\xi)))^{(|\alpha|)}\right|_{\xi=\xi_p}= |\sum_{\alpha}c_{\alpha}r_p^{|\alpha|}
\dfrac{\partial^{\alpha}(Q_{jp}(\tilde f_{ip}))}{\partial z_1^{\alpha_1}\dots\partial z_m^{\alpha_m}}(y_p + r_pu_p\xi_p)|\le C.r_p^{|\alpha|}M,
\end{align}
where $C,c_{\alpha}$ are suitable constants. Combine (\ref{m3}) and (\ref{m5}), we get
$$ (Q_{j,y_0}(\tilde g_i(\xi)))^{(|\alpha|)}|_{\xi=\xi_0}=\lim_{p\to \infty}(Q_{jp,y_p + r_pu_p\xi}(\tilde g_{ip}(\xi)))^{(|\alpha|)}_{\xi=\xi_p}=0$$
for $1\le |\alpha|\le m_j-1.$ Hence $\xi_0$ is a zero of $Q_{j,y_0}(\tilde g)$ with multiplicity at least $m_j$. Therefore, Claim 1 is proved. We know that
$G_p=\rho_{m_{\mathcal D}}(g_p)$  converges uniformly on compact subsets of $B(\xi_0, r_0)$ to $G=\rho_{m_{\mathcal D}}(g).$ Denote 
 $\tilde G$ by a reduced representaion of $G.$ Note that 
$Q_{j,y_0}^{\dfrac{m_{\mathcal D}}{d_j}}(\tilde g(\xi))=Q_j^{*}(y_0)(\tilde G(\xi)),$ then $G$ intersects $H_j^{*}(y_0)$ with multiplicity at least 
$\dfrac{m_{\mathcal D}}{d_j}m_j$ for each $1\le j\le q,$ where $H_j^{*}(y_0)$ is a associated hyperplane with linear form $Q_j^{*}(y_0)$ 
in $\mathbb P^{n_{\mathcal D}}(\mathbb C).$ Since $H_1^{*}(y_0), \dots, H_q^{*}(y_0)$ are in general position in
 $\mathbb P^{n_{\mathcal D}}(\mathbb C).$ From assumption,
$$\sum_{j=1}^{q}\dfrac{d_j}{m_j}<\dfrac{(q-n_{\mathcal D}-1)m_{\mathcal D}}{n_{\mathcal D}}$$
and apply Lemma \ref{lm5}, we get $G$ is a constant holomorphic map. Then $g$ is a constant holomorphic map from $\mathbb C$ into 
$\mathbb P^{n}(\mathbb C).$ This is a contradiction. Therefore, $\mathcal F$ is holomorphically normal family on $\mathbb D.$
\end{proof}

\begin{proof}[Proof of Theorem \ref{th11}]
By Lemma \ref{lm5a}, $Q_{jp}^{*}=Q_{jp}^{{m_{\mathcal D}}/{d_j}}:=Q_{j,f_p}^{{m_{\mathcal D}}/{d_j}}$ converge uniformaly on compact subset
 of $\mathbb D$ to $Q_j^{*}, 1\le j\le q,$ and 
$$ D(Q_1^{*}, \dots, Q_q^{*})(z)>\delta(z_0)>0$$
for some $z_0\in \mathbb D.$ In particular, the moving hyperplanes $H_1^{*}, \dots, H_q^{*}$ defining by moving linear forms $Q_1^{*}, \dots, Q_q^{*}$ 
(respectively) are located in weak general position in $\mathbb P^{n_{\mathcal D}}(\mathbb C).$ We recall that with writing both variables $z\in \mathbb D$ 
and $w\in \mathbb P^{n_{\mathcal D}}(\mathbb C)$ we have
$$ Q_{jp}^{*}(z)({\bf w})\to Q_j^{*}(z)({\bf w})$$
 uniformly on compact subsets in the variable $z\in \mathbb D.$ By Lemma \ref{lmb2} and Lemma \ref{lmb3}, after passing to a subsequence, we may 
assume that the sequence $\{f_p\}$ satisfies
$$f_p^{-1}(D_{k,f_p})=S_k \; (1\le k\le n_{\mathcal D}+1)$$
as a sequence of closed subsets of $\mathbb D,$ where $S_k$ are either empty or pure $(m-1)$-dimensional analytic sets in $\mathbb D,$ and satisfies
\begin{align}\label{cma}
\lim_{p\to \infty}\overline{\{z\in \text{supp}\nu_{Q_{k,f_p}(\tilde f_p)}| \nu_{Q_{k,f_p}(\tilde f_p)}(z)<m_k\}}-S=S_k\;(n_{\mathcal D}+2\le k\le q)
\end{align}
as a sequence of closed subsets of $\mathbb D-S,$ where $S_k$ are either empty or pure $(m-1)$-dimensional analytic sets in $\mathbb D-S.$
 
Let $T=(\dots, t_j, \dots) \; (1\le j\le q)$ be a family of variables. Set $\overset{\sim}Q_j^{*}=\sum_{j=0}^{n_{\mathcal D}}t_jw_j\in \mathbb Z[T,w].$ For 
each $L\subset \{1, \dots, q\}$ with $|L|=n_{\mathcal D}+1,$ take $\overset{\sim}R_L$ is the resultant of $\overset{\sim}Q_j^{*}, j\in L.$ Since 
$\{Q_j^{*}\}_{j\in L}$ are in weakly general position, then $\overset{\sim}R_L(\dots,c_{ij},\dots)\not\equiv 0.$ We set
$$ \overset {\sim}S:=\{z\in \mathbb D| \overset{\sim}R_L(\dots, c_{ij},\dots)=0\; \text{for some}\; L\subset\{1, \dots, q\}\; \text{with}\;
 |L|=n_{\mathcal D}+1 \}.$$
Let 
\begin{align}\label{cmb}
E=(\cup_{k=1}^{q}S_k\cup \overset{\sim} S)-S.
\end{align}
 Then $E$ is either empty or a pure $(m-1)$-dimensional analytic set in $\mathbb D-S.$  Fix any point
\begin{align}\label{cmc}
z_1\in (\mathbb D-S)-E.
\end{align}

 Choose a relatively compact neighborhood $U_{z_1}$ of $z_1$ in $(\mathbb D-S)-E.$ Then 
$\{f_p|_{U_{z_1}}\}\subset Hol(U_{z_1}, \mathbb P^n(\mathbb C)).$ We now prove that the family $\{f_p|_{U_{z_1}}\}$ is a holomorphically normal family. 
For this it is sufficient to show that the family $\{f_p|_{U_{z_1}}\}$ satisfies all conditions of Corollary \ref{corth10}. Indeed, there exists $N_0$ such that 
for all $p\ge N_0,$ $\{f_p|_{U_{z_1}}\}$ does not intersect with $D_{j,f_p}, 1\le j\le n_{\mathcal D}+1.$ From (\ref{cma})-(\ref{cmc}), we
 have $\{f_p|_{U_{z_1}}\}$ intersect with $D_{j,f_p}$ with multiplicity at least $m_j\; (n_{\mathcal D}+2\le j\le q).$ For all $z\in U_{z_1},$ we have 
$D(Q_1^{*}, \dots, Q_q^{*})(z)>0,$ if we take $m_j=\infty$ for all $1\le j\le n_{\mathcal D}+1,$ then all conditions of Corollary \ref{corth10} are satisfied. 
Thus, $\{f_p|_{U_{z_1}}\}$ is a holomorphically normal family. By the usual diagonal argument, we can find a subsequence (again denoted by $\{f_p\}$) 
which converges uniformly on compact subsets of $(\mathbb D-S)-E$ to a holomorphic mapping $f$ of $(\mathbb D-S)-E$ into $\mathbb P^{n}(\mathbb C).$  
Therefore $\{f_p\}$ is quasi-regular on $\mathbb D.$ Hence, $F$ is a quasi-normal family. The proof of Theorem \ref{th11} is completed.
\end{proof}
\begin{proof}[Proof Theorem \ref{th12} and Theorem \ref{th12a}]

In order to prove these theorems, we need some lemmas as follows:
\begin{lemma}\label{lmth12}
 Let $f$ be a holomorphic mapping from a bounded domain $\mathbb U$ in $\mathbb C^m$ into $\mathbb P^n(\mathbb C).$ Let $D_1, \dots, D_q$ 
be $q$ $(\ge 2n_{\mathcal D}+1)$ moving hypersurfaces in $\mathbb P^n(\mathbb C).$ Let 
$Q_j^{*}=\sum_{i=0}^{n_{\mathcal D}}d_{ij}\omega_i,\; (\omega_i={\bf x}^{I_i})$ in $H_{\mathbb D}[\omega_0,\dots,\omega_{n_{\mathcal D}}]$ be a
 homogeneous linearly defines the associated hyperplane $H_j^{*}$  of $D_j$ such that for any $z\in \overline {\mathbb U}:$
 $$D(Q_1^{*}, \dots, Q_q^{*})(z) > 0.$$ 
Let $\tilde f=(f_0,\dots, f_n)$ be a reduced representation of $f$  on $\mathbb U$ 
with $ f_i(z)\ne 0$ for some $i$ and $z,$ and when $m_j\geq 2$
        $$\sup_{1\leq |\alpha|\leq m_j-1, z \in f^{-1}(D_{j})}\Big|\dfrac{\partial^{\alpha}(Q_{j}(\tilde f_{i}))}{\partial z^{\alpha}}(z)\Big| < \infty$$
        hold for all $j\in\{1,\dots,q\},$  where $m_1, \dots, m_q$ are positive integers and may be $\infty,$ 
with $\sum_{j=1}^{q}\dfrac{d_j}{m_j}<\dfrac{(q-n_{\mathcal D}-1)m_{\mathcal D}}{n_{\mathcal D}}.$ 
Then $f$ is a normal holomorphic mapping from $\mathbb U$ into $\mathbb P^n(\mathbb C).$
\end{lemma}
\begin{proof}
Suppose that $f$ is not normal on $\mathbb U.$ Then, by Lemma \ref{nf}, 
  there exist $\{y_p\}\subset \mathbb U,$ $\{r_p\}$ with $r_p>0$ and $r_p\to 0^{+}$ and $\{u_p\}\subset\mathbb C^m$  are Euclidean unit vectors such that
$$ g_p(\xi):=f(y_p+r_pu_p\xi), \xi \in \mathbb C, $$
where $\lim_{p\to\infty}\dfrac{r_p}{d(y_p, \mathbb C^m\setminus \mathbb U)}=0,$ converges uniformly on compact subsets of $\mathbb C$ to a non-constant 
holomorphic mapping $g$ of $\mathbb C$ to $\mathbb P^{n}(\mathbb C).$ Since $\overline {\mathbb U}$ compact, then we may assume that
 $y_p\to y_0 \in \overline {\mathbb U}.$ We have  $G_p:=\rho_{m_{\mathcal D}}(g_p)$ converges uniformly on compact subsets of $\mathbb C$ to 
$G:=\rho_{m_{\mathcal D}}(g).$ By assumption, we have
$$ D(Q_1^{*}, \dots, Q_q^{*})(z)>0$$
for any $z\in \overline {\mathbb U}.$ In particular, the moving hyperplanes $H_1^{*}, \dots, H_q^{*}$ defining by  moving linear forms 
$Q_1^{*}, \dots, Q_q^{*}$ (respectively) are located in pointwise general position in $\mathbb P^{n_{\mathcal D}}(\mathbb C).$ 

There exist the reduced representations $\tilde{f}=(f_{0},...,f_{n})$ and $\tilde g=(g_0, \dots , g_n)$ of $f$ and $g$ respectively such that as 
$p\to\infty$, 
$\tilde{g}_p(\xi)= \tilde{f}(y_p + r_pu_p\xi)\to \tilde{g}(\xi) $ and hence
\begin{align}\label{m2}
 Q_{jp,y_p + r_pu_p\xi}(\tilde{g}_p(\xi) )\to Q_{j,y_0}(\tilde g(\xi))
\end{align}
converge uniformly on compact subsets of $\mathbb{C}$.

\noindent{\bf Claim}. $g(\mathbb{C})$ intersects with $D_{j,y_0}$ with multiplicity at least $m_j.$

Indeed, for $\xi_0\in \mathbb C$ with $Q_{j,y_0}(\tilde g(\xi_0))=0$, there exist a disc $B(\xi_0, r_0)$ in $\mathbb C$ and an index $i\in \{0,\dots, n\}$
 satisfying $
    g(B(\xi_0, r_0))\subset \{[x_0:...:x_n]\in \mathbb P^n(\mathbb C) : x_i\ne 0\}.$
Since $g_i(\xi)\ne 0$ on $B(\xi_0, r_0)$,  Hurwitz's theorem shows that $g_{ip}(\xi)=f_{i}(y_p + r_pu_p\xi)\ne 0$ on  $B(\xi_0, r_0)$ when $p$ is large
 enough. Thus $\tilde{g}_{ip}(\xi)=\left(\dfrac{g_{0p}}{g_{ip}},\dots, \dfrac{g_{np}}{g_{ip}}\right)\to \tilde g_i(\xi)$, and hence 
$Q_{jp,y_p + r_pu_p\xi}(\tilde g_{ip}(\xi))\to Q_{j,y_0}(\tilde g_i(\xi))$ converge uniformly on $B(\xi_0, r_0)$, which further implies that
\begin{align}\label{m3}
(Q_{jp,y_p + r_pu_p\xi}(\tilde g_{ip}(\xi)))^{(k)}\to (Q_{j,y_0}(\tilde g_i(\xi)))^{(k)}
\end{align}
converge uniformly on $B(\xi_0, r_0)$ for all $k\geq 1$.

By the assumption, when $m_j\geq 2,$
$$\sup_{1\leq |\alpha|\leq m_j-1, z \in f^{-1}(D_{j})}\Big|\dfrac{\partial^{\alpha}(Q_{j}(\tilde f_{i}))}{\partial z^{\alpha}}(z)\Big| < \infty$$
hold for all $j\in\{1,\dots,q\}$ and all $p\geq 1,$ where $f_{i}(z)\ne 0$  and $D_{j}$ is the moving hypersurface associated to $Q_{j}$. Since 
$Q_{j,y_0}(\tilde g(\xi_0))=0$, then Hurwitz's theorem shows that there exits $\xi_p\to \xi_0$ such that $ Q_{j,y_p + r_pu_p\xi_p}(\tilde{g}_p(\xi_p) )=0,$ 
which implies $y_p+ r_pu_p\xi_p\in f^{-1}(D_{jp}).$ Hence there exists $M>0$ such that when $m_j\geq 2\ (1\leq j\leq q)$ and $p$ is large enough
\begin{align*}
 |\dfrac{\partial^{\alpha}(Q_{j}(\tilde f_{i}))}{\partial z^{\alpha}}(y_p + r_pu_p\xi_p)|\le M
\end{align*}
for all $\alpha=(\alpha_1,\dots,\alpha_m): 1\le |\alpha| \le m_j-1$ and some $i\in \{0,\dots,n\},$ which means 
\begin{align}\label{m5}
\left |(Q_{j,y_p + r_pu_p\xi}(\tilde g_{ip}(\xi)))^{(|\alpha|)}\right|_{\xi=\xi_p}= |\sum_{\alpha}c_{\alpha}r_p^{|\alpha|}
\dfrac{\partial^{\alpha}(Q_{j}(\tilde f_{i}))}{\partial z_1^{\alpha_1}\dots\partial z_m^{\alpha_m}}(y_p + r_pu_p\xi_p)|\le C.r_p^{|\alpha|}M,
\end{align}
where $C,c_{\alpha}$ are suitable constants. Combine (\ref{m3}) and (\ref{m5}), we get
$$ (Q_{j,y_0}(\tilde g_i(\xi)))^{(|\alpha|)}|_{\xi=\xi_0}=\lim_{p\to \infty}(Q_{j,y_p + r_pu_p\xi}(\tilde g_{ip}(\xi)))^{(|\alpha|)}_{\xi=\xi_p}=0$$
for $1\le |\alpha|\le m_j-1.$ Hence $\xi_0$ is a zero of $Q_{j,y_0}(\tilde g)$ with multiplicity at least $m_j$. Therefore, Claim is proved. We know that
$G_p=\rho_{m_{\mathcal D}}(g_p)$  converges uniformly on compact subsets of $B(\xi_0, r_0)$ to $G=\rho_{m_{\mathcal D}}(g).$ Denote 
 $\tilde G$ by a reduced representaion of $G.$ Note that 
$Q_{j,y_0}^{\dfrac{m_{\mathcal D}}{d_j}}(\tilde g(\xi))=Q_j^{*}(y_0)(\tilde G(\xi)),$ then $G$ intersects $H_j^{*}(y_0)$ with multiplicity at least 
$\dfrac{m_{\mathcal D}}{d_j}m_j$ for each $1\le j\le q,$ where $H_j^{*}(y_0)$ is a associated hyperplane with linear form $Q_j^{*}(y_0)$ 
in $\mathbb P^{n_{\mathcal D}}(\mathbb C).$ Since $H_1^{*}(y_0), \dots, H_q^{*}(y_0)$ are in general position in
 $\mathbb P^{n_{\mathcal D}}(\mathbb C).$ From assumption,
$$\sum_{j=1}^{q}\dfrac{d_j}{m_j}<\dfrac{(q-n_{\mathcal D}-1)m_{\mathcal D}}{n_{\mathcal D}}$$
and apply Lemma \ref{lm5}, we get $G$ is a constant holomorphic map. Then $g$ is a constant holomorphic map from $\mathbb C$ into 
$\mathbb P^{n}(\mathbb C).$ This is a contradiction. Therefore, $f$ is normal function on $\mathbb U.$
\end{proof}

\begin{lemma}\label{lmth12a}\cite{NO} Let $M$ be a complex manifold  and let $S$ be a complex analytic subset of $M$ with $\text{codim}S\ge 2.$ 
Then $K_{M-S}=K_M$ on $M-S$ (i.e., the infinitesimal Kobayashi metric $K_{M-S}$ is the restriction of $K_M$ to $M-S$).
\end{lemma}
Now we prove Theorem \ref{th12}. Fix a point $P_0\in S,$ take a bounded neighborhood $U$ of $P_0$ in $\mathbb C^m$ (i.e., U is hyperbolic) with
 $\overline U \subset \mathbb D.$ By Lemma \ref{lmth12}, $f$ is normal holomorphic mapping from $U-S$ into $\mathbb P^{n}(\mathbb C).$ Thus by 
Defintion \ref{normal} and defintion of the integral distance, there exists a positive constant $c$ such that
$$ d_{\mathbb P^{n}(\mathbb C)}(f(z),f(w))\le cd_{U-S}^{K}(z, w) $$
for all $z,w\in U-S,$ where $d_{U-S}^{K}$ and $d_{\mathbb P^n(\mathbb C)}$ denote the Kobayashi distance on $U-S$ and the Fubini-Study distance
 on $\mathbb P^{n}(\mathbb C),$ resepctively. For any $z_0\in U\cap S,$ let $\{z_i\}_{i=1}^{\infty}$ be a sequence of points of $U-S$ so as to converge
 to $z_0.$ By Lemma \ref{lmth12a}, we have
$$ d_{\mathbb P^{n}(\mathbb C)}(f(z_i),f(z_j))\le cd_{U-S}^{K}(z_i, z_j)=cd_U^{K}(z_i,z_j).$$
Then $\{f(z_i)\}_{i=1}^{\infty}$ is a Cauchy sequence of $\mathbb P^{n}(\mathbb C)$ and hence $\{f(z_i)\}_{i=1}^{\infty}$ converges to a point 
$a_0\in \mathbb P^{n}(\mathbb C).$ Then $f(z)$ has an extension $h(z)$ on $U$ so as to be holomorphic on $U-S$ and continuous on $U$ and
 hence $h(z)$ is holomorphic on $U$ by Riemann's extension theorem. We have competed the proof of Theorem \ref{th12}.

Next, we prove Theorem \ref{th12a}. Fix a point $P_0\in S,$ take a bounded neighborhood $U$ of $P_0$ in $\mathbb C^m$ with
 $\overline U\subset \mathbb D.$ By Lemma \ref{lmth12}, $f$ is normal holomorphic mapping from $U-S$ into $\mathbb P^{n}(\mathbb C).$ By 
assumption of Theorem \ref{th12a}, $S$ is an analytic subset of a domain $U$ in $\mathbb C^m$ with codimension one, whose singularities are normal 
crossings. Hence $f(z)$ extends to a holomorphic mapping from $U$ into $\mathbb P^n(\mathbb C)$ by Theorem 2.3 in Joseph and Kwack \cite{JK}. We
 have finished the proof of Theorem \ref{th12}.
\end{proof}

\begin{proof}[Proof Theorem \ref{th13n}, Theorem \ref{th13} and Theorem \ref{th13a}]
The proof of Theorem \ref{th13n} is similarly to Theorem \ref{th10} by using Lemma \ref{lm5c} and Lemma \ref{lm6}. We omit it at here. In order to prove 
two next theorems, we need a lemma as follows:
\begin{lemma}\label{lmth13}
Let $f$ be a holomorphic mapping from $\mathbb U$ into $\mathbb P^n(\mathbb C).$ Suppose that there exist $n+1$ moving hypersurfaces $T_0, \dots, T_n$  in $\mathbb P^n(\mathbb C)$ of common degree which are defined by homogeneous polynomials $P_0,\dots,P_n\in H_{\mathbb D}[x_0,\dots,x_n]$ respectively, and that there exist $q$ moving hypersurfaces $D_1, \dots, D_q$ in $\overset{\sim}S(\{T_i\}_{i=0}^{n})$ such that the following conditions are satisfied:

\noindent $(i)$ $Q_j \in S(\{P_i\}_{i=0}^{n})$ is a homogeneous polynomial defined the $D_j,$ and for any $z\in \overline {\mathbb U},$
$$D(Q_1, \dots, Q_q)(z) > 0.$$

\noindent $(ii)$ Let $\tilde f=(f_0,\dots, f_n)$ be a reduced representation of $f$  on $\mathbb U$ 
with $ f_i(z)\ne 0$ for some $i$ and $z,$ and when $m_j\geq 2$
        $$\sup_{1\leq |\alpha|\leq m_j-1, z \in f^{-1}(D_{j})}\Big|\dfrac{\partial^{\alpha}(Q_{j}(\tilde f_{i}))}{\partial z^{\alpha}}(z)\Big| < \infty$$
        hold for all $j\in\{1,\dots,q\},$  where $m_1, \dots, m_q$ are positive integers and may be $\infty,$ 
with $$ \sum_{j=1}^{q}\dfrac{1}{m_j}<\dfrac{q-n-1}{n}.$$

Then $f$ is a normal holomorphic mapping from $\mathbb U$ into $\mathbb P^n(\mathbb C).$
\end{lemma}
\begin{proof}
The proof of Lemma \ref{lmth13} is similarly to Lemma \ref{lmth12}. For convenience the reader, we show it detail. Suppose that $f$ is not normal on
$\mathbb U.$ Then, by Lemma \ref{nf}, there exist $\{y_p\}\subset {\mathbb U},$ $\{r_p\}$ with $r_p>0$ and $r_p\to 0^{+}$ and $\{u_p\}\subset\mathbb C^m$ are Euclidean unit vectors such that
$$ g_p(\xi):=f(y_p+r_pu_p\xi), \xi \in \mathbb C, $$
where $\lim_{p\to\infty}\dfrac{r_p}{d(y_p, \mathbb C^m\setminus \mathbb U)}=0,$ converges uniformly on compact subsets of $\mathbb C$ to a 
non-constant holomorphic mapping $g$ of $\mathbb C$ to $\mathbb P^{n}(\mathbb C).$ Since $\overline {\mathbb U}$ compact, then we may assume that
 $y_p\to y_0 \in \overline {\mathbb U}.$ We have $G_p:=F(g_p)$ converges uniformly on compact subsets of $\mathbb C$ to $G:=F(g),$ where 
$$ F:\mathbb P^{n}(C)\to \mathbb P^{n}(\mathbb C)$$
is the mapping defined by
$$F_i({\bf x}) = P_i({\bf x})\;  (0 \le i \le n).$$
 Furthermore, we can write $Q_j$ as follows
$$Q_j =\sum_{i=0}^{n} b_{ij}P_i,\; b_{ij}\in H_{\mathbb U}, j=1, \dots, q.$$
Set 
$$ \mathcal H_j=\{x\in \mathbb P^n(\mathbb C)|\sum_{i=0}^{n}b_{ij}x_j=0\}.$$
By Lemma \ref{lm7},  from
$$ D(Q_1, \dots, Q_q)(z)>0$$
for any $z\in D,$ we have $\{P_i\}_{i=0}^{n}$ are in general position and $\{Q_j\}_{j=1}^{q}$ are located in general position in $S(\{P_i\}_{i=0}^{n}).$ 
This means that the hyperplanes $\{\mathcal H_j\}_{j=1}^{q}$ are located in general pointwise position in $\mathbb P^{n}(\mathbb C).$

By arguments as Lemma \ref{lmth12}, $G$ intersects $\mathcal H_{j}(y_0)$ with multiplicity at least $m_j$ for each $1\le j\le q.$  From assumption,
$$\sum_{j=1}^{q}\dfrac{1}{m_j}<\dfrac{q-n-1}{n}$$
and apply Lemma \ref{lm5}, we get $G=F(g)$ is a constant holomorphic map. Then $g$ is a constant holomorphic map from $\mathbb C$ into
 $\mathbb P^{n}(\mathbb C)$ by Lemma \ref{lm6}. This is a contradiction. Therefore, $f$ is a  normal holomorphic mapping on $\mathbb U.$
\end{proof}
Apply Lemma \ref{lmth13}, we can prove Theorem \ref{th13} and Theorem \ref{th13a} similarly as Theorem \ref{th12} and Theorem \ref{th12a}. Hence, 
we omit them here.

\end{proof}
\begin{proof}[Proof of Theorem \ref{th14}]
Suppose that $F$ is not normal on $\mathbb D.$ Then, by Lemma \ref{lm5c}, there exists a subsequence denoted by $\{f_p\}\subset F$ and 
$y_0\in K_0$, $\{y_p\}_{p=1}^{\infty}\subset K_0$ with $y_p\to y_0,$ $\{r_p\}\subset (0, +\infty)$ with $r_p\to 0^{+},$ and 
$\{u_p\}\subset \mathbb C^m$, which are unit vectors, such that $g_p(\xi) := f_p(y_p + r_pu_p\xi)$ converges uniformly on compact subsets of $\mathbb C$
 to a nonconstant holomorphic map $g$ of $\mathbb C$ into $\mathbb P^{n}(\mathbb C).$  By Lemma \ref{lm5a}, $Q_{jp}:=Q_{j,f_p}$ 
converge uniformaly on compact subset of $\mathbb D$ to $Q_j, 1\le j\le q,$ and 
$$ D(Q_1, \dots, Q_q)(z)>\delta(z)>0$$
for any fixed $z\in \mathbb D.$ In particular, the moving hypersurfaces $D_1, \dots, D_q$ defining by moving homogeneous polynomial  $Q_1, \dots, Q_q$ 
(respectively) are located in pointwise general position in $\mathbb P^{n}(\mathbb C).$ 
We recall that with writing both variables $z\in \mathbb D$ and  $x \in \mathbb P^{n}(\mathbb C)$ we have
$$ Q_{jp}(z)({\bf x})\to Q_j(z)({\bf x})$$ uniformly on compact subsets in the variable $z\in \mathbb D.$

By arguments as the proof of Theorem \ref{th10}, the map $g$ intersects $D_{j}(y_0)$ with multiplicity at least $m_j.$ Now, we show that $g$ is a constant 
mapping. Since $Q_1(y_0), \dots, Q_q(y_0)$ are in general position in $\mathbb P^{n}(\mathbb C),$ then there are at most $n$ hypersurfaces in the family 
$D_1(y_0),\dots,D_q(y_0)$ containing image of $g.$ We denote $\mathcal I=\{i\in \{1,\dots, q\}:Q_i(y_0)(\tilde g)\equiv 0\}.$ Therefore, we must have 
$0\le |\mathcal I|\le n.$ Note that $q\ge n n_{\mathcal D}+2n+1,$ then  $q-|\mathcal I|\ge n n_{\mathcal D}+n+1.$ We see that $g$ intersects with
 $\{D_j(y_0)\}_{j\not\in \mathcal I}$ with multiplicity at least $m_j,$ and 
\begin{align*}
 \sum_{j\not\in \mathcal I}\dfrac{1}{m_j}&\le \sum_{j=1}^{q}\dfrac{1}{m_j}<\dfrac{q-n-(n-1)(n_{\mathcal D}+1)}{n_{\mathcal D}(n_{\mathcal D}+2)}\\
&\le \dfrac{q-|\mathcal I|-(n-1)(n_{\mathcal D}+1)}{n_{\mathcal D}(n_{\mathcal D}+2)}.
\end{align*}
Apply Lemma \ref{lm9} for map $g$ and hypersurfaces $\{D_j(y_0)\}_{j\not\in I},$ we get $g$ is a constant map. This is a contradiction. Therefore, $F$ 
is holomorphically normal family on $\mathbb D.$
\end{proof}

\begin{proof}[Proof of Theorem \ref{th15}]
By Lemma \ref{lm5a}, $Q_{jp}:=Q_{j,f_p}$ converge uniformaly on compact subset of $\mathbb D$ to $Q_j, 1\le j\le q,$ and 
$$ D(Q_1, \dots, Q_q)(z)>\delta(z_0)>0$$
for some $z_0\in \mathbb D,$ where $z$ belongs a neighborhood of $z_0.$ In particular, the moving hypersurfaces $D_1, \dots, D_q$ defining by moving homogeneous polynomial  $Q_1, \dots, Q_q$ 
(respectively) are located in weak general position in $\mathbb P^{n}(\mathbb C).$ We recall that with writing both variables $z\in D$ and 
$x\in \mathbb P^{n}(\mathbb C)$ we have
$$ Q_{jp}(z)({\bf x})\to Q_j(z)({\bf x})$$
 uniformly on compact subsets in the variable $z\in \mathbb D.$ By Lemma \ref{lmb2} and Lemma \ref{lmb3}, after passing to a subsequence, 
we may assume that the sequence $\{f_p\}$ satisfies
$$f_p^{-1}(D_{k,f_p})=S_k \; (1\le k\le n+1)$$
as a sequence of closed subsets of $\mathbb D,$ where $S_k$ are either empty or pure $(m-1)$-dimensional analytic sets in $\mathbb D,$ and satisfies
\begin{align}\label{cma1}
\lim_{p\to \infty}\overline{\{z\in \text{supp} \nu_{Q_{k,f_p}(\tilde f_p)}| \nu_{Q_{k,f_p}(\tilde f_p)}(z)<m_k\}}-S=S_k\;(n+2\le k\le q)
\end{align}
as a sequence of closed subsets of $\mathbb D-S,$ where $S_k$ are either empty or pure $(m-1)$-dimensional analytic sets in $\mathbb D-S.$
 
Let $T=(\dots, t_{jI}, \dots) \; (1\le j\le q, |I|=d_j)$ be a family of variables. Set 
$\overset{\sim}Q_j=\sum_{|I_j|=d_j}t_{jI}{\bf x}^{I_j}\in \mathbb Z[T,{\bf x}].$ Suppose that
$$Q_j({\bf x})=Q_j(x_0, \dots, x_n) = \sum\limits_{k=0}^{n_{d_j}}e_{kI_{j,k}} x_0^{i_{kj,0}}\dots x_n^{i_{kj,n}},$$
where $I_{j,k}=(i_{kj,0},\dots,i_{kj,n}), |I_{j,k}| =d_j$ for $k=0,\dots,n_{d_j}$ and $n_{d_j}= \binom{n+d_j}{n}-1.$ 
For each $L\subset \{1, \dots, q\}$ with $|L|=n+1,$ take $\overset{\sim}R_L$ is the resultant of $\overset{\sim}Q_j, j\in L.$ Since $\{Q_j\}_{j\in L}$ 
are in weakly general position, then $\overset{\sim}R_L(\dots,e_{kI_{j,k}},\dots)\not\equiv 0$ (see \cite{VW}, chapter 16). Set
$$ \overset {\sim}S:=\{z\in D| \overset{\sim}R_L(\dots, e_{kI_{j,k}},\dots)=0\; \text{for some}\; L\subset\{1, \dots, q\}\; \text{with}\; |L|=n+1 \}.$$
Let 
\begin{align}\label{cmb1}
E=(\cup_{k=1}^{q}S_k\cup \overset{\sim} S)-S.
\end{align}
 Then $E$ is either empty or a pure $(m-1)$-dimensional analytic set in $\mathbb D-S.$  Fix any point
\begin{align}\label{cmc1}
z_1\in (D-S)-E.
\end{align}

 Choose a relatively compact neighborhood $U_{z_1}$ of $z_1$ in $(\mathbb D-S)-E.$ Then
 $\{f_p|_{U_{z_1}}\}\subset Hol(U_{z_1}, \mathbb P^n(\mathbb C)).$ We now prove that the family $\{f_p|_{U_{z_1}}\}$ is a holomorphically normal family.
  For this it is sufficient to show that the family $\{f_p|_{U_{z_1}}\}$ satisfies all conditions of Corollary \ref{corth14}. Indeed, there exists $N_0$ such that for all $p\ge N_0,$ $\{f_p|_{U_{z_1}}\}$ does not intersect with $D_{j,f_p}, 1\le j\le n+1.$ From (\ref{cma1})-(\ref{cmc1}), we have $\{f_p|_{U_{z_1}}\}$ intersect with $D_{j,f_p}$ with multiplicity at least $m_j\; (n+2\le j\le q).$ 
For all $z\in U_{z_1},$ we have $D(Q_1, \dots, Q_q)(z)>0,$ if we take $m_j=\infty$ for all $1\le j\le n+1,$ then all conditions of Corollary \ref{corth14} are 
satisfied. Thus, $\{f_p|_{U_{z_1}}\}$ is a holomorphically normal family. By the usual diagonal argument, we can find a subsequence 
(again denoted by $\{f_p\}$) which converges uniformly on compact subsets of $(\mathbb D-S)-E$ to a holomorphic mapping $f$ of $(\mathbb D-S)-E$ into
 $\mathbb P^{n}(\mathbb C).$  By Lemma \ref{lm5b}, $\{f_p\}$  is meromorphically convergent on $\mathbb D.$ Hence, $F$ is meromorphicaly normal 
family on $\mathbb D.$ The proof of Theorem \ref{th15} is completed.
\end{proof}

\begin{proof}[Proof Theorem \ref{th16} and Theorem \ref{th17}]
In order to prove these theorems, we need some lemmas as follows:
\begin{lemma}\label{lm18}
 Let $f$ be a holomorphic mapping from a bounded domain $\mathbb U$ in $\mathbb C^m$ into $\mathbb P^n(\mathbb C).$ Let $\mathcal D=\{D_1, \dots, D_q\}$ be $q$ $(\ge nn_{\mathcal D}+2n+1)$ moving hypersurfaces in $\mathbb P^n(\mathbb C).$ Let $Q_j\in H_{\mathbb D}[x_0,\dots,x_n]$ be a homogeneous polynomial with degree $d_j$ which is defined $D_j$ such that for any $z\in \overline {\mathbb U}:$
 $$D(Q_1, \dots, Q_q)(z) > 0.$$ 
Let $\tilde f=(f_0,\dots, f_n)$ be a reduced representation of $f$  on $\mathbb U$ 
with $ f_i(z)\ne 0$ for some $i$, then when $m_j\geq 2$
        $$\sup_{1\leq |\alpha|\leq m_j-1, z \in f^{-1}(D_{j})}\Big|\dfrac{\partial^{\alpha}(Q_{j}(\tilde f_{i}))}{\partial z^{\alpha}}(z)\Big| < \infty$$
        hold for all $j\in\{1,\dots,q\},$  where $m_1, \dots, m_q$ are positive integers and may be $\infty,$ 
with $$\sum_{j=1}^{q}\dfrac{1}{m_j}<\dfrac{q-n-(n-1)(n_{\mathcal D}+1)}{n_{\mathcal D}(n_{\mathcal D}+2)}.$$ 
Then $f$ is a normal holomorphic mapping from $\mathbb U$ into $\mathbb P^n(\mathbb C).$
\end{lemma}

Lemma \ref{lm18} is proved similarly Lemma \ref{lmth13} and Theorem \ref{th14}. We omit it at here. Theorem \ref{th16} and Theorem \ref{th17} are 
proved similarly Theorem \ref{th12} and Theorem \ref{th12a}, we omit them at here.
\end{proof}


\begin{thebibliography}{x}

\bibitem{GK} \textsc{G. Aladro and S. G. Krantz,} {\it A  criterion  for normality  in  $\C^n,$}  J. Math. Anal. Appl,
\textbf{161}(1991), 1-8.
\bibitem{CA} \textsc{H. Cartan,}   \textit{Sur les zeros des combinaisions linearires de $p$ fonctions holomorpes donnees,} 
Mathematica (Cluj).  7, 80-103, 1933.


\bibitem{DTT} \textsc{G. Dethloff, D. D. Thai and P. N. T. Trang}, {\it Normal families of meromorphic mappings of several complex variables for moving hypersurfaces in a complex projective space,} Nagoya Math. J. 217, 23-59, 2015.


\bibitem{F} \textsc{H. Fujimoto,} {\it Extensions of the big Picard's theorem,} Tohoku Math. J. 24, 415-422, 1972.
\bibitem{F1} \textsc{H. Fujimoto,} {\it On families of meromorphic maps into the complex projective space,}
Nagoya Math. J. 54, 2151, 1974.

\bibitem{HT} \textsc{P. C. Hu and N. V. Thin}, {\it Generalizations of Montel's normal criterion and Lappan's five-valued theorem to holomorphic curves}, 
To appear in Complex Var. Elliptic Equ. 2019.

\bibitem{JK} \textsc{J. Joseph and M. Kwack,} {\it Extension and convergence theorem for families of normal maps in several complex variables,} Proc. Amer. Math. Soc. 125, 1675-1684, 1997.

\bibitem{MG} \textsc{M. Green,} {\it Holomorphic maps into complex projective space omitting hyperplanes,} Trans. Amer.
Math. Soc. 169, 89-103, 1972.
\bibitem{EN} \textsc{E. Nochka,} {\it On the theory of meromorphic functions,} Soviet Math. Dokl. 27, 377-381, 1983.

\bibitem{N} \textsc{J. Noguchi,} {\it A note on entire pseudo-holomorphic curves and the proof of Cartan-Nochka's theorem,}
 Kodai Math. J. 28, 336-346, 2005.

\bibitem{NO} \textsc{J. Noguchi and T. Ochiai} {\it Geometry function theory in several complex variables,} Transl. Math. Monogr. 80, Amer. Math. Soc, Providence, RI, 1990.

\bibitem{M1} \textsc{ P. Montel,} {\it Sur les familles de fonctions analytiques qui admettent des valeurs exceptionnelles dans un domaine}, Ann. Sci. Éc. Norm. Supér. 29, 487-535, 1912.
\bibitem{M2} \textsc{P. Montel,} {\it Sur les familles normales de fonctions analytiques}, Ann. Éc. Norm. Supér. 33, 223-302, 1916.

\bibitem{QA1} \textsc{S. D. Quang and D. P. An,} {\it Second Main Theorems for meromorphic mappings with moving hypersurfaces and a uniqueness 
problem,} Comput. Methods Funct. Theory. 17(3), 445-461, 2017.

 \bibitem{ST} \textsc{W. Stoll,} {\it Normal families of non-negative divisors,} Math. Z. 84 (1964), 154-218.


 \bibitem{T1} \textsc{Z. Tu,} {\it Normality criterions for families of holomorphic mappings of several complex variables into $\mathbb P^{N}(\mathbb C)$,} 
Proc. Amer. Math. Soc. 127, 1039-1049, 1999.

\bibitem{T2} \textsc{Z. Tu}, {\it On meromorphically normal families of meromorphic mappings of several
complex variables into $\mathbb P^{N}(\mathbb C)$,} J. Math. Anal. Appl. 267, 1-19, 2002.

\bibitem{TL} \textsc{Z. Tu and P. Li,} {\it Normal families of meromorphic mappings of several complex variables into $\mathbb P^{N}(\mathbb C)$ for 
moving targets,} Sci. China Ser. A.  48, 355-364, 2005.
\bibitem{TL1} \textsc{Z. Tu and P. Li,} {\it Big Picard's theorems for holomorphic mappings of several variables into $\mathbb P^n(\mathbb C)$ with 
moving hyperplanes,} J. Math. Anal. Appl.  342, 629-638, 2006.

 \bibitem{VW} \textsc{V. D. Waerden,} Algebra, vol 2, Spinger-Verlag, New York, 2003.

\end{thebibliography}
\end{document}